\documentclass[12pt]{article}

\usepackage{amsmath,amsthm,amssymb,amsfonts}
\usepackage{hyperref}

\usepackage[a4paper, total={6in, 8in}]{geometry}

\usepackage[T1]{fontenc}\usepackage[upright]{fourier}\usepackage{baskervald}
\usepackage[french, english]{babel}

\usepackage{dsfont}

\usepackage{tikz-cd}
\usetikzlibrary{arrows}

\usepackage[normalem]{ulem}

\usepackage{verbatim}

\theoremstyle{plain}
\makeatletter
\newtheorem*{rep@theorem}{\rep@title}
\newcommand{\newreptheorem}[2]{%
\newenvironment{rep#1}[1]{%
 \def\rep@title{#2 \ref{##1}}%
 \begin{rep@theorem}}%
 {\end{rep@theorem}}}
\makeatother

\newtheorem{thm}{Theorem}[section]
\newtheorem{lem}[thm]{Lemma}
\newtheorem{prop}[thm]{Proposition}
\newtheorem{cor}[thm]{Corollary}

\theoremstyle{definition}
\makeatletter
\newtheorem*{cont@example}{\cont@title}
\newcommand{\newcontexample}[2]{%
\newenvironment{cont#1}[1]{%
 \def\cont@title{#2 \ref{##1} continued}%
 \begin{cont@example}}%
 {\end{cont@example}}}
\makeatother

\newtheorem{defn}[thm]{Definition}

\newtheorem{exmp}[thm]{Example}

\newtheorem{rem}[thm]{Remark}

\theoremstyle{remark}

\newtheorem*{notation}{Notation}
\newtheorem*{assumption}{Assumption}

\newtheorem*{SF}{Serre's Generation and Finiteness Theorems}
\newtheorem*{CBC}{Cohomology and Base Change} 
\newtheorem*{Quot}{Representability of the Quot Functor} 
 
\newtheorem*{Flattening}{Existence of the Universal Flattening Stratification} 
\newtheorem*{CM}{Theorem of Castelnuovo and Mumford}
\newtheorem*{GR}{Gotzmann Regularity}
\newtheorem*{CPS}{The Cohomology of Projective Space}
\newtheorem*{SCT}{Semicontinuity Theorem}
\newtheorem*{STI}{Scheme Theoretic Image}
\newtheorem*{HPRF}{Hilbert Polynomials and Relative Flatness}

\newreptheorem{thm}{Theorem}
\newreptheorem{lem}{Lemma}
\newcontexample{exmp}{Example}
\newreptheorem{prop}{Proposition}

\begin{document}

\title{Quasi-Splines and their Moduli}
\author{Patrick Clarke}
\date{\today}

\maketitle

\begin{abstract}
We study what we call {\bf quasi-spline sheaves} over locally Noetherian schemes.
This is done with the intention of considering splines from the point of view of moduli theory.  In other words, we study the way in which certain objects 
that arise
in the theory of splines can be made to depend  on parameters.
In addition to  {\bf quasi-spline sheaves}, we treat
{\bf ideal difference-conditions}, and individual {\bf quasi-splines}.
Under certain hypotheses each of these types of objects admits
a fine moduli scheme.  The moduli of quasi-spline sheaves is proper, and
there is a natural compactification of the moduli of ideal difference-conditions.
We include some speculation on the uses of these moduli in the theory of splines
and topology, and  an appendix with a treatment of the {\bf Billera-Rose homogenization}
in scheme theoretic language.
\end{abstract}

\selectlanguage{french}
\begin{abstract}
On \'etudie ce que l'on appelle {\bf les faisceaux de quasi-splines} sur des sch\'emas localement noeth\'erien, l'id\'ee \'etant
de les consid\'erer du point de vue de la th\'eorie des espaces de modules.
En d'autres termes, on \'etudie la fa\c{c}on dont certains objets issus de la th\'eorie des splines peuvent d\'ependre des param\`etres.  En plus des {\bf faisceaux de quasi-splines}, on \'etudie les {\bf conditions-diff\'erences d'id\'eaux} et les {\bf quasi-splines} individuelles.
Sous certaines hypoth\`eses, chacun de ces types d'objets admet un sch\'ema de modules fin.
On d\'emontre que le sch\'ema de modules des faisceaux quasi-splines est propre
et qu'il existe une compactification naturelle de l'espace
de modules des conditions-diff\'erences d'id\'eaux.  On discute
finalement de l'utilisation qui pourrait \^etre faite de ces espaces de modules en th\'eorie des splines et en topologie.  L'article inclut 
une annexe o\`u
{\bf l'homog\'en\'eisation de Billera-Rose} est pr\'esent\'ee dans le langage de la th\'eorie des sch\'emas.
\end{abstract}
\selectlanguage{english}

\section{Introduction}
\subsection{Quasi-spline sheaves}
Given a scheme $B$, the $s$-fold sum $\mathcal{O}_B^s$ 
is a sheaf of $\mathcal{O}_B$-algebras
with entrywise multiplication and addition.
A quasi-coherent subsheaf $\mathcal{S}$ of $\mathcal{O}_B^s$
generalizes the notion of spline functions on $B$ if it contains the diagonal 
copy of $\mathcal{O}_B \subseteq \mathcal{O}_B^s$ and is closed 
under multiplication.  This is simply to say $\mathcal{S}$ is a quasi-coherent
 $\mathcal{O}_B$-subalgebra of  $\mathcal{O}_B^s.$
Because of the close relationship with spline functions, we
call such an $\mathcal{S}$ a sheaf of {\bf quasi-splines}.  
A simple example is 
Example \ref{example:simplest}.
\begin{exmp} 
\label{example:simplest}
Let $B = \operatorname{Spec}\mathbf{R}[x].$
The sheaf associated to the 
$\mathbf{R}[x]$-module
$$
S = \{(g_1,  g_2) \ | \ g_1 - g_2 \in (x^2) \} \subseteq (\mathbf{R}[x])^2
$$
is a sheaf of quasi-splines. It is naturally  thought of as the splines with continuous first derivatives 
over the subdivision  $\mathbf{R} = (-\infty, 0] \cup [0, \infty).$
\end{exmp}
We will focus on quasi-splines over projective schemes as in Example \ref{exmp:saturated}.
\begin{exmp} 
\label{exmp:saturated}
Let $B = \mathbf{P}^1_\mathbf{R} = \operatorname{Proj}\mathbf{R}[x, z].$
The sheaf $\mathcal{S}$ associated to the graded 
$\mathbf{R}[x, z]$-module
$$
 ^h  S = \{(G_1,  G_2) \ | \ G_1 - G_2 \in (x^2) \} \subseteq (\mathbf{R}[x, z])^2.
$$
is a sheaf of quasi-splines.  $ ^h  S$ is the homogenization of the splines of Example \ref{example:simplest}
as defined in \cite{billera-rose:1991}.  Additionally, $ ^h  S$ is saturated, i.e. the map 
$$ ^h  S \to \bigoplus_d \Gamma(\mathbf{P}^1_\mathbf{R}, \mathcal{S}(d))$$
is a graded isomorphism.  Together, these facts imply the module from Example \ref{example:simplest}
is canonically identified as $S = \mathcal{S}(U_0)$ where $U_0 \subseteq \mathbf{P}^1_\mathbf{R}$
is the set on which $z \neq 0.$
\end{exmp}

Although  quasi-splines are closely related to splines.  It is not always possible to think of them as such.
Consider  Example \ref{exmp:not-spline}.
\begin{exmp} 
\label{exmp:not-spline}
Let $B = \operatorname{Spec}\mathbf{R}[x].$
The sheaf associated to the 
$\mathbf{R}[x]$-module
$$
S = \{(g_1,  g_2) \ | \ g_1 - g_2 \in (x^2+1) \} \subseteq (\mathbf{R}[x])^2
$$
is a sheaf of quasi-splines, but it cannot be thought of as splines in any obvious way.
\end{exmp}

We are interested in studying quasi-spline sheaves which depend of 
parameters.  To this end,
given a $Z$-scheme $B$
we define a {\bf $Z$-family of quasi-spline sheaves over $B$} 
as
\begin{itemize}
\item a sheaf of quasi-splines $\mathcal{S}$
over   $B$ such that 
\item for any morphism $f\colon Z' \to Z$, the pullback 
$\pi_B^*\mathcal{S}$
is a sheaf of quasi-splines over $Z' \times_Z B.$
\end{itemize}
This definition eliminates from consideration sheaves 
$\mathcal{S} \subseteq \mathcal{O}_{B}^s$
whose inclusion map
$\mathcal{S}\to  \mathcal{O}_{B}^s$
fails to be an inclusion after fixing the value of the parameters.  Example \ref{exmp:family-fail}
gives a sheaf of quasi-splines which fails to be a family.
\begin{exmp}
\label{exmp:family-fail}
Let $Z = \operatorname{Spec} \mathbf{R}[z],$ 
and take $B =  
\operatorname{Spec}\mathbf{R}[x]$
as in Example \ref{example:simplest}.
 The morphism $B \to Z$ is given by $z \mapsto x.$
The sheaf associated to the 
$\mathbf{R}[x]$-module
$$
S = \{(g_1,  g_2) \ | \ 
g_1-g_2 \in (x^2) \} \subseteq (\mathbf{R}[x])^2
$$
is not a $Z$-family. This can be seen by first setting $g = (-x^2, x^2)$
and identifying $$S = \mathbf{R}[x][g] / (g^2 -x^4).$$
The map $S \to (\mathbf{R}[x])^2$ sends $a_0 + a_1 g \mapsto (a_0 - a_1x^2, a_0+a_1x^2).$
So when $z = 0$ we have $x = 0$  and $S|_{z=0} = \mathbf{R}[g]/(g^2).$
The map to $(\mathbf{R}[x])^2|_{z=0} \cong \mathbf{R}^2$ is not an inclusion since it sends $g \mapsto 0.$
\end{exmp}
On the other hand, Example \ref{exmp:interesting} shows some quasi-spline sheaves are indeed families.
\begin{exmp} 
\label{exmp:interesting}
Let $Z = \operatorname{Spec} \mathbf{R}[z],$ and
$B = \operatorname{Spec} \mathbf{R}[z][x, y].$
The sheaf associated to the 
$\mathbf{R}[z][x,y]$-module
$$
S = \{(g_1,  g_2, g_3) \ | \ 
\begin{array}{l}
g_1-g_2 \in (x), \\
g_2-g_3 \in (y), \\
g_1-g_3 \in (x+y-z)
 \end{array}\} \subseteq (\mathbf{R}[z][x,y])^3.
$$
is a $Z$-family. One can check this by observing that as an $\mathbf{R}[z]$-module, $S$ has a free basis 
whose $\mathbf{R}$-span is 
$$
\mathbf{R}[x,y] \cdot v_0 \oplus \mathbf{R}[x,y] \cdot v_1 \oplus  \mathbf{R}[x,y] \cdot v_2 \oplus \mathbf{R}[y] \cdot v_3
$$
where 
$$
\begin{array}{rclrrrl}
v_0 & = & (& 1,& 1,& 1& ),\\
v_1 & = & (& 0,& zx - x^2, & zx - x^2 - xy & ),\\
v_2 & = & (& 0,& 0,&  zy - xy - y^2 & ), \text{ and}\\
v_3 & = & (& 0,& xy,& 0& ).
\end{array}
$$
\end{exmp}

The definition of families of quasi-spline sheaves
guarantees that if we fix a scheme $Y$ over $T$, the assignment 
$$
\mathcal{QS}^{(s)}(Y/T)(Z) = \{ \text{$Z$-families of quasi-spline sheaves $ \mathcal{S} \subseteq \mathcal{O}^s_{Z \times_T Y}$} \}
$$
is functoral for $T$-schemes $Z$. This means there is some hope 
that one can find a representing scheme, i.e. 
there is a {\bf moduli scheme}
${QS^{(s)}(Y/T)} \in T$-schemes
such that $\operatorname{Mor}(Z, {QS^{(s)}(Y/T)}) = \mathcal{QS}^{(s)}(Y/T)(Z).$
Our first theorem is on the existence of this moduli scheme.
\begin{repthm}{thm:main} In the category of locally Noetherian schemes,
for a 
flat, projective $T$-scheme $Y$ 
the functor
$$
\mathcal{QS}^{(s)}(Y/T) \colon (T-\text{\bf schemes})^\text{op} \to \text{\bf Sets} 
$$
is representable by a closed subscheme  ${QS^{(s)}(Y/T)}$ of the Quot scheme 
$\operatorname{Quot}(\mathcal{O}_Y^s/Y/T).$
\end{repthm}

\subsection{Ideal difference-conditions}
In many of the applications we have in mind, $\mathcal{S}$  is  defined 
as the subset of $\mathcal{O}_Y^s$ whose sections satisfy {\bf ideal difference-conditions}.
That is to say, the sheaf $\mathcal{S}$ is {\bf defined by conditions} that
written locally are
\begin{equation}
\label{equation:ideal-difference}
\mathcal{S} = \{ (g_1, \dotsc, g_s) \ | \ g_j- g_k \in \mathfrak{I}_{jk} \text{ for } 1 \leq j < k \leq s \}
\end{equation}
for ${s \choose 2}$ ideals $\mathfrak{I}_{jk}\subseteq \mathcal{O}_Y.$
All of our examples defined quasi-splines this way.

Allowing the ideals to vary by introducing parameters leads to an interesting subtlety. 
For a fixed value of the parameters, there are two different ways to define a quasi-spline sheaf.
On one hand
we can compute the sheaf of 
quasi-splines 
defined by the ideals with the parameters considered as 
variables, and then restrict the sheaf to the fixed parameter values.
On the other, we could fix value of the parameters in the ideals
and then compute a possibly different quasi-spline sheaf.

To be clear, denote by 
$\mathcal{S}_\mathfrak{I}$
the sheaf of quasi-splines defined by ideals $(\mathfrak{I}_{jk})_{jk}.$
For simplicity assume that we have a single parameter $z \in \mathbf{R}[z],$ and we are interested in the fixed value $z=0.$
Consider ideals $(\mathfrak{I}_{jk}(z))_{jk}$
which depend on  $z$.
There is a natural map 
\begin{equation}
\label{equation:natural-map}
\mathcal{S}_{\mathfrak{I}(z)}|_{z=0} \to \mathcal{S}_{\mathfrak{I}(z=0)}
\end{equation}
which may or may not be an isomorphism.  However, 
the map is an inclusion for all $z$ if and only 
if $\mathcal{S}_{\mathfrak{I}(z)}$ is a $Z$-family of quasi-spline sheaves (here $Z= \operatorname{Spec} \mathbf{R}[z]$).

The ideals in Example \ref{exmp:interesting} are shown in the continuation of this example 
to lead to 
 sheaves where the map in (\ref{equation:natural-map}) is
an inclusion but not an isomorphism at $z=0.$ 
\begin{contexmp}{exmp:interesting}
For any given $z \in \mathbf{R}$, the quasi-splines of Example \ref{exmp:interesting} are naturally thought of as splines over the region $\Omega$ of plane in the complement 
of the triangle with vertices
$(z , 0), (0, z),$ and $(0,0).$  The relevant subdivision is shown in Figure \ref{figure:3-regions}, and is made up of three parts
\begin{itemize}
\item $\Omega_1 = \{ (x,y) \ | \ 0 \leq x, 0 \leq x+y -z \},$ 
\item $\Omega_2 = \{ (x,y) \ | \ x \leq 0 \leq y \},$  and
\item $\Omega_3 = \{ (x,y) \ | \ x +y -z \leq 0, y \leq 0 \}.$ 
\end{itemize}
The sheaf defined by first setting $z=0$ and then computing quasi-splines is strictly larger than those obtained by 
restricting from the family.  For instance,
$$
(y, y-x,-x)
$$
is a quasi-spline for the $z=0$ ideals, but it is not the restriction of a quasi-spline in the family.  In
other words, the map in (\ref{equation:natural-map}) is 
an inclusion but is not surjective.

As in Example \ref{example:simplest} these two sets of splines can be characterized in terms of continuity and the existence of derivatives. The splines
in the family when restricted to $z=0$ are exactly those which are both continuous over $\Omega$
and have continuous first partial derivatives at $(0,0) \in \mathbf{R}^2$.  The splines computed by first setting $z=0$ is the larger set of all continuous splines.
\end{contexmp}
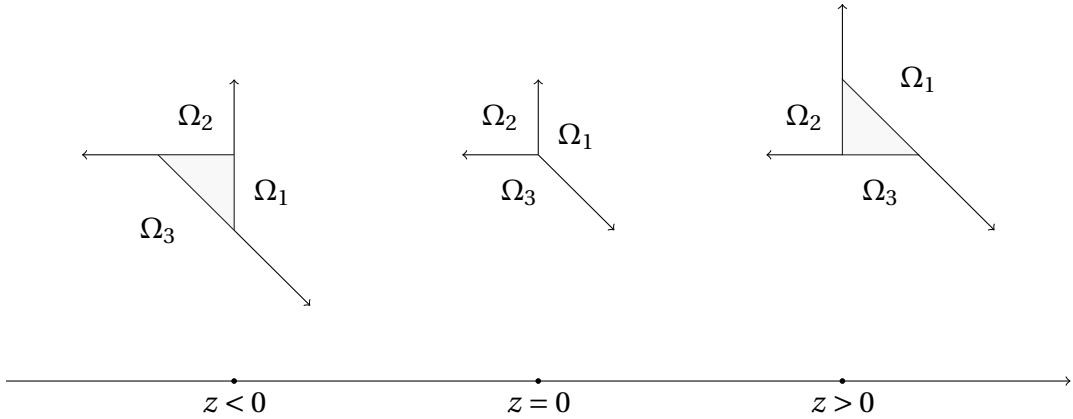
\begin{figure}[h!]
\begin{center}
\begin{tikzpicture}[scale=1]
\coordinate (n0) at (-3-1,0);
\coordinate (nX1) at (-2-1,0);
\coordinate (nXn1) at (-4-1,0);
\coordinate (nYn1) at (-3-1,-1);
\coordinate (nXn2) at (-5-1, 0);
\coordinate (nY1) at (-3-1,1);
\coordinate (nY2) at (-3-1,2);
\coordinate (nQ2) at (-1-1,-1);
\coordinate (nQQ2) at (-2-1, -2);
\coordinate (nO1) at (1/2-3-1,-1/2); 
\coordinate (nO2) at (-1/2-3-1,1/2); 
\coordinate (nO3) at (-1-3-1,-1); 

\draw [->] (n0) -- (nXn2); 
\draw [->] (n0)-- (nY1);
\draw [->](nYn1)--(nQQ2);
\draw (n0) -- (nYn1);
\draw (nYn1) -- (nXn1);
\fill[gray!20,nearly transparent] (n0) -- (nXn1) -- (nYn1)--cycle;
\node at (nO1){$\Omega_1$};
\node at (nO2){$\Omega_2$};
\node at (nO3){$\Omega_3$};

\coordinate (z0) at (0,0);
\coordinate (zX1) at (1,0);
\coordinate (zXn1) at (-1,0);
\coordinate (zYn1) at (0,-1);
\coordinate (zXn2) at (-2, 0);
\coordinate (zY1) at (0,1);
\coordinate (zY2) at (0,2);
\coordinate (zQ2) at (1,-1);
\coordinate (zQQ2) at (1, -2);
\coordinate (zO1) at (1/2,1/2/2); 
\coordinate (zO2) at (-1/2,1/2); 
\coordinate (zO3) at (-1/2/2,-1/2); 

\draw [->] (z0) -- (zXn1); 
\draw [->] (z0)-- (zY1);
\draw [->] (z0)--(zQ2);
\node at (zO1){$\Omega_1$};
\node at (zO2){$\Omega_2$};
\node at (zO3){$\Omega_3$};

\coordinate (p0) at (3+1,0);
\coordinate (pX1) at (4+1,0);
\coordinate (pXn1) at (2+1,0);
\coordinate (pY1) at (3+1,1);
\coordinate (pY2) at (3+1,2);
\coordinate (pQ2) at (5+1,-1);
\coordinate (p2Q2) at (7+1, -1);
\coordinate (pO1) at (1+3+1,1); 
\coordinate (pO2) at (-1/2+3+1,1/2); 
\coordinate (pO3) at (1/2+3+1,-1/2); 


\draw (p0) -- (pX1) -- (pY1)--cycle;
\fill[gray!20,nearly transparent] (p0) -- (pX1) -- (pY1)--cycle;

\draw [->] (pY1) -- (pY2);
\draw [->] (pX1) -- (pQ2);
\draw [->] (p0) -- (pXn1);
\node at (pO1){$\Omega_1$};
\node at (pO2){$\Omega_2$};
\node at (pO3){$\Omega_3$};

\coordinate (tp) at (7,-3);
\coordinate (t0) at (0,-3);
\coordinate (tn) at (-7, -3);

\draw[->](tn)--(tp);
\draw[fill] (t0) circle (0.03);
\node [below]at (t0) {$z=0$};
\draw[fill] (-4,-3) circle (0.03);
\node [below]at (-4,-3) {$z<0$};
\draw[fill] (4,-3) circle (0.03);
\node [below]at (4,-3) {$z>0$};

\end{tikzpicture}
\end{center}
\caption{$\Omega$ and the subdivision as $z$ varies from Example 
\ref{exmp:interesting} and its continuation.
}
\label{figure:3-regions}
\end{figure}

With these considerations in mind, for a $Z$-scheme $B$
we define a
 {\bf $Z$-family of ideal difference-conditions over $B$}
 to be an ${s \choose 2}$-tuple $(\mathfrak{I}_{jk})_{jk}$ of
quasi-coherent ideals $\mathfrak{I}_{jk} \subseteq \mathcal{O}_B$
which have two properties:
\begin{itemize}
\item $\mathfrak{I}_{jk}$ remains an ideal after any base change, i.e. $\mathbf{V}(\mathfrak{I}_{jk})$
is flat, and
\item base change of the quasi-splines defined by the $\mathfrak{I}_{jk}$'s equals the 
quasi-splines defined by the base change of the  $\mathfrak{I}_{jk}$'s.
\end{itemize}
Denoting the Hilbert scheme by  $\operatorname{Hilb}(Y/T),$ we have the representability theorem:
\begin{repthm}{thm:condition-moduli}
In the category of locally Noetherian schemes,
for any  projective $T$-scheme $Y$ 
the functor
$$
\mathcal{C}^{(s)}(Y/Z)(Z) = \{ \text{$Z$-families of ideal difference-conditions on $Z \times_T Y$} \}
$$
is representable by a scheme ${C^{(s)}(Y/Z)}$ obtained as the universal 
flattening stratification of a certain sheaf $\mathcal{H}$ over
$\operatorname{Hilb}(Y/T)^{s \choose 2}.$ 
\end{repthm}
\noindent Notice that there is no $T$-flatness condition on $Y$ in this statement.  
This has the interesting consequence that $C^{(s)}(Y/Z)$ exists, even if $QS^{(s)}(Y/Z)$ might not.

The moduli of ideal difference conditions is not proper.  This can been seen in  Example \ref{exmp:interesting}.
This example  suggests a {\bf compactification} of $C^{(s)}(Y/T).$  To do this
we use an auxiliary scheme 
based on the notion of a {\bf compatible pair} $(\mathcal{S}, (\frak{I}_{jk})_{jk}).$
By this we mean
the composition  
$$
 \mathcal{S} \to \mathcal{O}_B^s \to \bigoplus_{jk} \mathcal{O} / \mathfrak{I}_{jk}
$$
is zero.  Equivalently, $\mathcal{S}$ is contained in the quasi-spline sheaf $\mathcal{S}_\mathfrak{I}$ made up of sections 
of $\mathcal{O}^s_B$ which satisfy the ideal difference-conditions $(\mathfrak{J}_{jk})_{jk}$.  However, it is not necessary that $\mathcal{S}$ equals $\mathcal{S}_\mathfrak{I}$.

Based on this idea, we construct {the \bf moduli of compatible pairs} $P^{(s)}(Y/T).$   This scheme is proper and in the category of locally Noetherian schemes it represents the functor
$$
\mathcal{P}^{(s)}(Y/T)(Z) =\{ \text{compatible pairs $(\mathcal{S}, (\mathfrak{I}_{jk})_{jk})$ of $Z$-families over $Z\times_TY$}  \}.
$$
${C^{(s)}(Y/Z)}$  sits as a locally closed subscheme of ${P^{(s)}(Y/T)}$
and presented this way, a natural compactification
is given by the  scheme-theoretic closure
 $\overline{{C}}^{(s)}(Y/T)$  of ${C^{(s)}(Y/Z)}$ in ${P^{(s)}(Y/T)}.$ 

An important property of the compactification is that it allows for a {\bf universal family of quasi-spline sheaves}
 over $\overline{{C}}^{(s)}(Y/T) \times_T Y$ that extends the one naturally living over ${C^{(s)}(Y/T)} \times_T Y.$  
  The case for the ``correctness'' of this choice of 
 compactification can be made on the grounds of Proposition \ref{prop:C-correct}:
 \begin{repprop}{prop:C-correct}
 Let ${C^{(s)}(Y/Z)} \to H$ be a morphism to a scheme such that $H \times_Y Y$ is
equipped 
with an $H$-family of compatible pairs whose restriction to ${C^{(s)}(Y/Z)} \times_T Y$ 
is the universal family
of ideal difference-conditions.
Assume $H$ equals the scheme theoretic image of ${C^{(s)}(Y/Z)}$ in $H$.  Then 
the morphism $H \to {P^{(s)}(Y/T)}$ factors through $\overline{{C}}^{(s)}(Y/T).$
 \end{repprop}
 \noindent
 However,
 a possibly more compelling fact is the naturality of the families $\overline{{C}}^{(s)}(Y/T)$ admits, such as the one in  Example \ref{exmp:interesting}.
 
 On the moduli of ideal-difference conditions the Hilbert polynomial of the quasi-spline sheaf is locally unchanged.
 It is natural to ask for a further stratification of the moduli space into subschemes on which the full Hilbert series is unchanged.
 Our section on ideal difference-conditions concludes with a discussion on how  the {\bf degeneracy loci} 
 of a morphism of certain locally free sheaves can be used to give such a stratification.
 
\subsection{Quasi-splines}
In the construction of the moduli space ${QS^{(s)}(Y/T)}$ we assumed that $Y$ was flat. 
A consequence of this is that the Hilbert polynomial of $\mathcal{S}$ 
is locally independent of the point in ${QS^{(s)}(Y/T)}.$  Using this fact, we get another interesting
theorem about representing the functor of sections:
\begin{repthm}{thm:sections}
Fix flat $T$-scheme $Y$, and for notational simplicity, assume that Hilbert polynomial $p_{\mathcal{O}_Y}$ of $Y$ is independent of $T$.
In the category of locally Noetherian schemes,
the functor 
 $$
\mathcal{E}^{(s)}_{p, d}(Y/T)(Z) =  \{ \tau \in \Gamma(Z \times_T Y, \mathcal{S}(d)) \ | \ \text{ $\mathcal{S}$ has Hilbert polynomial $p$} \}
$$
is representable for $d \geq m$ where $m$ depends on $p$ and $p_{\mathcal{O}_Y}$.  
\end{repthm}

Denote the universal quotient associated to the Quot scheme by $\mathcal{G}.$
Over the piece of ${QS^{(s)}(Y/T)}$ which lies in the component of the Quot scheme labeled by the polynomial $p_\mathcal{G}=  s \, p_{\mathcal{O}_Y} - p,$
the representing scheme is
$$
{E_{p,d}^{(s)}(Y/T)} = \operatorname{{Spec}} \operatorname{Sym} \mathcal{V}_d,
$$
where $$\mathcal{V}_d = \mathcal{H}om_{{QS^{(s)}(Y/T)}}(\pi_*\mathcal{S}(d), \mathcal{O}_{{QS^{(s)}(Y/T)}}),$$
 and $\pi \colon {QS^{(s)}(Y/T)} \times_T Y \to {QS^{(s)}(Y/T)}$ is the projection.
The  number $m$ can be taken to be the maximum the Gotzmann numbers of 
the pair of polynomials  given below in Lemma \ref{lem:hilbert-polynomials}. 

A fact of independent interest used in the proof  is that for $d$ at or beyond this value, the sheaf  $\pi_*\mathcal{S}(d)$ is locally free.  This implies that the rank of $\pi_*\mathcal{S}(d)$ agrees with its Hilbert polynomial.

\subsection{Billera-Rose Homogenization}

The paper concludes with an appendix on the  homogenization procedure introduced in 
\cite{billera-rose:1991}.  Originally, this was an identification between 
splines on a triangulation in $\mathbf{R}^n$ with splines on the cone over the triangulation
in $\mathbf{R}^{n+1}.$  The splines over the cone form a graded module, and the degree $d$ homogeneous
piece of this module is naturally identified with splines on $\mathbf{R}^n$ all of whose entries
are degree $\leq d.$

We consider this procedure as a comparison between quasi-spline sheaves on three schemes: the original scheme $A,$ its projective closure $\widehat{A},$
and the affine cone over the projective closure. 
We find 
in Proposition \ref{prop:saturation} that if the homogeneous coordinate ring $\widehat{A}$
is a quotient of the homogeneous coordinate ring of the ambient projective space, then
Billera-Rose homogenization translates questions about quasi-splines on the original scheme into questions about an quasi-splines on
its projective closure.

To prove this result, homogenization and projective closure is formulated in terms of filtered algebras and modules, rather than the traditional approach of 
submodules of graded modules \cite{EGA_II}.  We find that this approach is very satisfying and interesting in its own right.

\subsection{Remarks and Speculations}

\subsubsection{Complementary techniques}

Our moduli spaces complement a larger line of investigation into multivariate splines, and in the hope of facilitating reciprocity between this work and the existing research programmes, we sketch out some of the basics of these alternate approaches.  We do not use the techniques that are typically used in research on splines, but we expect that this difference will prove to be an asset.  The object of study is essentially the same and results discovered from one point of view can be used to inform the other.  

 For the most part, current investigations begin with a given class of  triangulations (or polyhedral complexes) in $\mathbf{R}^n$
over which they consider piecewise polynomials.  This set-up appeared in the original spline literature of
 Hrennikoff \cite{hrennikoff-1941}, Courant \cite{courant-1943} and Schoenberg \cite{schoenberg-1946}.  In recent work this basic view is enhanced by advanced techniques such as  
 the so-called B\'ezier-Bernstein 
methods 
and tools from homological and commutative algebra.

{ \bf B\'ezier-Bernstein methods.}
B\'ezier-Bernstein methods were first 
used by de Casteljau  \cite{deCasteljau-1959}
 and then reintroduced in Farin \cite{farin-1977}.  They have proven extremely valuable in the theory of splines as evidenced by 
 the Hilbert polynomial computations in  Alfeld-Schumaker \cite{alfeld-schumaker-1987, alfeld-schumaker-1990} and Alfeld-Schumaker-Whiteley \cite{alfeld-schumaker-whiteley}.
 
These methods are based on expansion of splines in Mob\"ius's barycentric coordinates \cite{mobius:1827}.  These
are functions 
$$
(\mu_0, \dotsc, \mu_n) \colon \Delta^n \to \mathbf{R}^{n+1}
$$
which embed $\Delta^n$ into $\mathbf{R}^{n+1}$ as the subset 
$$
\Delta^n = \{ (p_0 , \dotsc, p_n)  \in\mathbf{R}^{n+1}  \ | \ p_0 + \dotsm +p_n=1   \text{ and } p_i \geq 0 \ \forall i \}.
$$
 The restriction of a spline on a triangulation to a simplex can be expanded as a 
 polynomial in the $\mu$'s.   Thus a spline can be encoded in a list of polynomials in the $\mu$'s: one for each 
 $n$-simplex in a triangulation.
 
 The characteristic feature of B\'ezier-Bernstein methods is to consider the splines in their $B$-form. 
 This means a degree $d$ is fixed, 
 and splines whose degree is bounded by $d$ are represented by a list of 
homogeneous degree $d$ 
polynomials written as a span of the  normalized monomials 
  $$
 b_{(\nu_0, \dotsc, \nu_n)} = \frac{(\nu_0+ \dotsm +\nu_n)!}{\nu_0! \dotsm \nu_n!} \mu_0^{\nu_0} \dotsm \mu_n^{\nu_n}
 $$
 called the barycentric Bernstein polynomials.  When $\nu_0+ \dotsm +\nu_n = d$
 these polynomials form a basis for the degree $\leq d$ polynomials on $\Delta^n.$
 This  makes good use of the 
 the seemingly unfortunate fact that 
 the relation $1 = \mu_0+ \dotsm + \mu_n$ leads to many expansions for a given polynomial.
 
 There are two distinct advantages that the B-form representation of a spline provides.  The first is that 
 the normalization guarantees 
 $$
 \sum_{\nu_0+\dotsm +\nu_n = d}  b_{(\nu_0, \dotsc, \nu_n)}  = (\mu_0+ \dotsm + \mu_n)^d = 1.
 $$
 and so the approximation argument of Bernstein's proof \cite{bernstein-1912} 
 of 
 Weierstrauss's approximations theorem \cite{weierstrass-1885}
 can be immediately adapted.  The other advantage is that  given a pair of $n$-simplicies in a triangulation 
 which share a facet, one can easily check if a polynomial assignment to the pair defines 
 continuous function.  For example, if the shared facet is the $0^\text{th}$ and the vertex order agrees 
 on it, then the two polynomials must agree when $\mu_0 = 0.$  This amounts to checking equality of 
 the coefficients in the $B$-form of those  barycentric Bernstein polynomials with $\nu_0 = 0.$

Additional aspects of B\'ezier-Bernstein methods include de Casteljau's algorithm \cite{deCasteljau-1959}
that treats the computational problem of evaluating a spline given it its $B$-form as a function 
 in the usual coordinates on $\mathbf{R}^n.$ A related problem is understanding how the $B$-form changes 
 under barycentric subdivision of the simplices.  A comprehensive reference for this approach to splines is the book of Lai-Schumaker \cite{lai-schumaker}.

 {\bf Homological Algebra.}  Closer to the spirit of our approach are those which use tools from commutative  and homological algebra.
 Homological algebraic thinking appeared  as early as Schumaker \cite{schumaker:1979}, 
and was used explicitly in Billera's proof \cite{billera-1988} of Strang's conjecture \cite{strang:1974}
 on the dimension of splines spaces.  Specifically, Schumaker considered ideal-difference conditions and  the first terms of a complex 
 fully introduced by  Billera. The homology in degree zero is the ring of spline
 functions.  This complex was refined by Scheck \cite{schenck:1997} who produced another complex with splines in degree zero, and has the interesting property that  the module of splines is flat as a module over the ring of polynomials if and only the first cohomology is zero.  
 
 These and subsequent investigations introduced tools from commutative algebra and combinatorics, such as local cohomology \cite{schenck-stillman:1997}, Gr\"obner bases \cite{billera-rose:1989}, and  posets \cite{yuzvinsky:1992}.  An additional interesting participant is the theory of hyperplane arrangements  as found in Schenk's proof \cite{schenck-2014} of a conjecture of Foucart-Sorokina \cite{foucart-sorokina-2013}. 
  
 {\bf Geometry.} Geometry itself has been used too.  Stiller \cite{stiller:1983} identified splines over certain subdivisions in the plane with global sections of certain vector bundles over $\mathbf{P}^1.$  This identification was exploited by using Riemann-Roch to produce explicit formulas. This point of view was developed in several papers such as 
 Iarrobino \cite{iarrobino-1997}, Schenck-Germita \cite{geramita-schenk-1997} and
 Schenck-Stiller \cite{Schenck-Stiller-2002}.  In a different direction, Yuzvinsky \cite{yuzvinsky:1992}  considerations of a \v{C}ech resolution of splines over a polyhedral complex  is decidedly geometric.
 
Deep connections between the geometric picture and the B\'ezier-Bernstein methods can be seen in the Hilbert polynomial formulas of Alfeld-Schumaker \cite{alfeld-schumaker-1987, alfeld-schumaker-1990} and Alfeld-Schumaker-Whiteley \cite{alfeld-schumaker-whiteley}.   These are expressed in terms of incidence conditions between different facets of the triangles in the given triangulation.  One can see immediately in these  the  ancient geometric technique of Appollonius 
 now understood as specifying a linear system in terms of base-points.

Together, these various viewpoints on spline functions give a lot of information about the the moduli scheme, the relevant  degeneracy loci and Fitting subschemes.  We expect that as we learn more about its geometric and arithmetic properties, these will also serve to enrich these other approaches to the subject.

\subsubsection{The questions of dimension and flatness}
The constructions here are particularly suitable to  the {\bf dimension question} in the theory of splines.
This was posed by   
Strang \cite{strang:1974}, and in this context asks
\begin{center}
``What is the Hilbert series of $\mathcal{S}?$''
\end{center}
The Hilbert polynomial of $\mathcal{S}$ does not change as one moves around within connected components of ${C^{(s)}(Y/T)}.$  This means that just knowing the connected component determines most of the Hilbert series.  

In general, the problem of determining the initial terms of the Hilbert series is daunting. However, 
 when cohomology commutes with base change for $\mathcal{O}_{{C^{(s)}(Y/T)} \times_T Y}$ and  the 
$\mathcal{O}_{{C^{(s)}(Y/T)} \times_T Y}/\mathfrak{I}_{jk}$ (e.g. hypersurface ideal difference-conditions on $\mathbf{P}^n$) the geometry governing the rank of 
$\Gamma(Y, \mathcal{S}(d))$ for small $d$ is the stratification of ${C^{(s)}(Y/T)}$ defined by the 
{\bf degeneracy loci} of the map 
$$
\Gamma(Y, \mathcal{O}^s_{C^{(s)}(Y/T)}(d)) \to 
\bigoplus_{jk} \Gamma(Y, \mathcal{O}_{{C^{(s)}(Y/T)} \times_T Y}(d)/\mathfrak{I}_{jk}(d)).
$$
This is proved  below in Proposition \ref{prop:degeneracy-loci}.
Ultimately, the dimension question for small $d$ is a question of understanding how these subschemes lie in ${C^{(s)}(Y/T)}.$

A related question concerns the {\bf flatness} of the splines over $Y$: 
\begin{center}
``Which quasi-spline sheaves  are flat $\mathcal{O}_Y$-modules?'' 
\end{center}
This question was posed by Billera and Rose in the context of the dimension question \cite{billera-rose-modules}.
This can be interpreted in terms of the $s^\text{th}$ {\bf Fitting subscheme} $Z_s$ of the universal quasi-spline sheaf on 
${C^{(s)}(Y/T)} \times_T Y.$ 
Recall that $Z_s$ is the closed subscheme over which $\mathcal{S}$ cannot be generated by  $s$-sections.
The image
of $Z_s$ under the projection 
$\pi \colon {C^{(s)}(Y/T)} \times_T Y \to {C^{(s)}(Y/T)}$ is made up of those quasi-spline sheaves 
which are not flat on $Y$.

\subsubsection{Spline domains and approximation strategies}

The existence of these moduli spaces points to some interesting possibilities in approximation theory.  For instance in an approximation or interpolation problem, rather than fixing a sheaf $\mathcal{S}$ of quasi-splines and trying to find a best candidate in $\Gamma(Y, \mathcal{S}(d)),$  one could consider the problem of finding a best quasi-spline in $E_{p,d}^{(s)}(Y/T).$  In principle, this frees one from committing to a fixed {\bf spline domain} $D:\Omega = \Omega_1 \cup \dotsc \cup \Omega_s \subseteq Y(\mathbf{R})$, and allows the subdivision to vary.

Putting this onto a satisfactory mathematical footing would require a moduli of spline-domains $\mathcal{D}.$  One could then consider compatible triples
$$(D, (\mathfrak{I}_{jk})_{jk}, \tau) \in \mathcal{D} \times_T \overline{C}^{(s)}(Y/T) \times_T {E_{p,d}^{(s)}(Y/T)}.$$ 
We know of no such object $\mathcal{D}$ in the literature, but see no reason why it shouldn't exist.  Some insight is provided by Example \ref{exmp:moduli-loop} which
indicates the sort of phenomena that arise when interpreting quasi-splines as splines.
\begin{exmp}
\label{exmp:moduli-loop}
Consider $f_z(x,y) = (z^2-1)y - z(x^2 + y^2-1)$ as a family of polynomials on $\mathbf{R}^2$
parameterized by $z \in [-1,1].$  For each $z$ write $\Omega = \Omega_1^z \cup \Omega_2^z \subseteq \mathbf{R}^2$
where 
\begin{itemize}
\item $\Omega_1^z = \{ (x,y) \ | \ f_z(x,y) \leq 0 \}$, and
\item $\Omega_2^z = \{ (x,y) \ | \ f_z(x,y) \geq 0 \}.$
\end{itemize}
Consider the family of splines defined by the quasi-spline $$g_z = (-zf_z(x,y), zf_z(x,y)).$$  Observe that at both $z=-1$ and $z=1$ the quasi spline is $(x^2 + y^2-1, 1-x^2 - y^2).$  However,  $\Omega_1$ and $\Omega_2$
have switched, 
so the spline has reversed signs.  This is illustrated in Figure \ref{figure:sign-reversal}.
Topologically, this is an interval with distinct endpoints in 
$\mathcal{D} \times {E_{p,d}^{(s)}(Y/T)}$ whose projection to ${E_{p,d}^{(s)}(Y/T)}$ is a loop.
\begin{figure}[h]
\begin{center}
\begin{tikzpicture}[scale=1]
\coordinate (n0) at (-3,0);
\draw (n0) circle  (1);
\node at (n0) {$\  \Omega_1$};
\node at (-2,-1.5) {$\Omega_2$};

\coordinate (p0) at (3,0);
\draw (p0) circle (1); 
\node at (p0) {$\  \Omega_2$};
\node at (4,-1.5) {$\Omega_1$};

\coordinate (tp) at (6,-3);
\coordinate (t0) at (0,-3);
\coordinate (tn) at (-6, -3);

\draw[->](tn)--(tp);
\draw[loosely dotted](-6,3)--(6,3);
\draw[fill] (-3,-3) circle (0.03);
\node [below]at (-3,-3) {$z=-1$};
\draw[fill] (3,-3) circle (0.03);
\node [below]at (3,-3) {$z=1$};

\begin{scope}[shift={(-3,3)}]
\draw[thick,domain=-1:1] plot (\x, {\x*\x-1}) ;
\draw[thick,domain=-3/2:-1] plot (\x, {1-\x*\x}) ;
\draw[thick,domain=1:3/2] plot (\x, {1-\x*\x})node[right] {$g$}  ;
\end{scope}

\begin{scope}[shift={(3,3)}]
\draw[thick,domain=-1:1] plot (\x, {1-\x*\x}) ;
\draw[thick,domain=-3/2:-1] plot (\x, {\x*\x-1}) ;
\draw[thick,domain=1:3/2] plot (\x, {\x*\x-1}) node[right] {$g$} ;
\end{scope}

\end{tikzpicture}
\end{center}
\caption{
The splines and spline domains at the ends of the interval $[-1,1]$ from Example \ref{exmp:moduli-loop}.
The spline is indicated by its graph when restricted to $y=0.$ These two functions on $\mathbf{R}^2$
are defined by the same quasi-spline.
}
\label{figure:sign-reversal}
\end{figure}
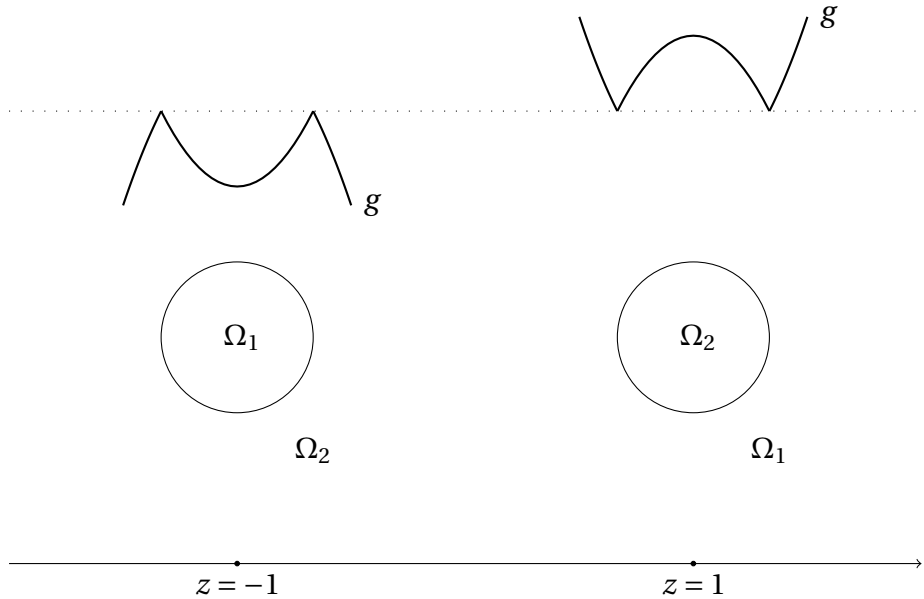
\end{exmp}

\subsubsection{Topology}

In addition to the close relationship to spline theory, quasi-splines have been singled out   in equivariant cohomology and equivariant intersection homology \cite{gilbert-polster-tymoczko} under the name {\bf generalized splines}.  This is the part of the program which began with a description of  the equivariant cohomology smooth compactification of an algebraic group in terms of splines by Bifet-De Concini-Procesi \cite{bifet-deconcini-procesi}.   Brion \cite{brion-1997}  extended this to certain singular spaces, and the most general setting in which quasi-splines appear seems to be the equivariantly formal spaces of Goresky-Kottwitz-MacPherson
\cite{goresky-kottwitz-macpherson-1998}.  These note worthy points in these investigations are  Payne \cite{payne-2006} and Schenck \cite{schenck-2011}.
For us, this  opens up a huge area of connections to topics such as geometric representation theory, Schubert calculus, and quantum cohomology.

\section{Assumptions, Conventions and Notations} 
 This paper is written in scheme-theoretic language.   In this section we collect several relevant standard results, make notations, 
 and specify our assumptions.  These breakdown roughly as notations for projective geometric constructions, results relevant to cohomology and base change, and finally the representability of certain functors such as flattening stratifications and Quot schemes.

\begin{assumption}
We fix an integer $s \geq 1$ throughout.
We are working a category of locally Noetherian schemes, and if we are over a base scheme, this scheme
is also locally Noetherian.  We fix schemes $Y$ and $T.$  This allows us to simplify our notation.  
For example, $QS^{(s)}(Y/T)$ will be written $QS.$ 
\end{assumption}

\begin{notation}
If $B \to Z$ is a $Z$-scheme, $\mathcal{F}$ is a sheaf on $B$, and $\phi \colon Z'\to Z$ a morphism, we denote
\begin{itemize} 
\item the fiber product $B_{Z'} = B \times_Z Z',$ and
\item the pullback $\phi^*\mathcal{F}$ on $B_{Z'}$ by $\mathcal{F}|_{Z'}.$
\end{itemize}
A point $q \in Z$ is assigned the scheme structure $\operatorname{Spec} \mathbf{k}(q),$ and
we often write $B_q$ and $\mathcal{F}|_q$ with this scheme structure on $q$ assumed.
The vertical bar  in the notation for the pullback is to avoid confusion with the stalk $\mathcal{F}_b$
of $\mathcal{F}$ at a point $b \in B.$
\end{notation}

\subsection{Projective Geometry}
We review here some basic constructions and facts of projective geometry.  This is done mostly to establish notation.

\begin{SF} (\cite{serre:1955} see also \cite[Theorem II.5.17]{hartshorne-1977}) Let $Z$ be a Noetherian scheme and $\mathcal{F}$ a coherent sheaf on a projective 
$Z$-scheme $\pi \colon B \to Z.$ Then
\begin{itemize}
\item $R^i\pi_*\mathcal{F}$ is a coherent $\mathcal{O}_Z$-module, and
\end{itemize}
for all sufficiently large $d$
\begin{itemize}
\item $\mathcal{F}(d)$ is generated by global sections, and
\item $R^i\pi_*\mathcal{F}(d) = 0.$ 
\end{itemize}
\end{SF}

\begin{notation}
For a sheaf of graded modules $N$ over a sheaf of graded $\mathcal{O}_Z$-algebras $\mathcal{R}$, we write  
$\tilde{N}$ for the associated sheaf on $\operatorname{{Proj}}(\mathcal{R}).$
Conversely, given a sheaf $\mathcal{F}$ on $\operatorname{{Proj}}(\mathcal{R}),$
we write 
$\Gamma_{*}(\mathcal{F})$ for the graded $\Gamma_*(\mathcal{O}_{\operatorname{{Proj}}(\mathcal{R})})$-module
$$
\Gamma_{*}(\mathcal{F}) = \bigoplus_{d } \pi_* \mathcal{F}(d),
$$
and  $\Gamma_{\geq m}(\mathcal{F})$ if we only take those  $d \geq m.$
Here $\pi \colon \operatorname{{Proj}}(\mathcal{R}) \to Z$ is the projection.
\end{notation}

\begin{lem}
The following statements describe the relationship between $\Gamma_*,$ $\Gamma_{\geq m}$
and $\tilde{(\cdot)}:$
\begin{itemize}
\item $\tilde{(\cdot)}$ is exact;
\item $\Gamma_*$ and $\Gamma_{\geq m}$ are left exact; 
\item $\tilde{(\cdot)} \circ \Gamma_{\geq m}= \tilde{(\cdot)} \circ \Gamma_{*};$
\item $\tilde{(\cdot)}$
 left adjoint to $\Gamma_*;$
 \item the counit $\epsilon \colon \tilde{(\cdot)} \circ \Gamma_*  \to \mathbf{1}$ is a  natural isomorphism;
\item the unit $\eta \colon \mathbf{1} \to  \Gamma_* \circ \tilde{(\cdot)}$ is called the {\bf saturation map.}
 \end{itemize}
\end{lem}
\begin{proof}
Omitted.
\end{proof}

\begin{rem}
In the affine case, $\tilde{(\cdot)}$ is used for the functor assigning to a module over a ring the associated sheaf on  
the spectrum. Its adjoint equivalence is  $\Gamma(\cdot).$
\end{rem}

\begin{CPS}
(\cite[Theorem III.5.1]{hartshorne-1977}) Let $Z$ be an affine Noetherian scheme.    Then:
\begin{itemize}
\item the natural map $\mathcal{O}_Z[x_0, \dotsc, x_n] \to \Gamma_*(\mathcal{O}_{\mathbf{P}_Z^n})$ is an isomorphism of graded $\mathcal{O}_Z[x_0, \dotsc, x_n]$-modules,
\item $H^i(\mathbf{P}_Z^n, \mathcal{O}_{\mathbf{P}_Z^n}(d)) = 0$ for $0 < i < n$ and all $d$,
\item $H^n(\mathbf{P}_Z^n, \mathcal{O}_{\mathbf{P}_Z^n}(-n-1)) \cong \mathcal{O}_Z$, and
\item the natural map $H^0(\mathbf{P}_Z^n, \mathcal{O}_{\mathbf{P}_Z^n}(d)) \times
H^n(\mathbf{P}_Z^n, \mathcal{O}_{\mathbf{P}_Z^n}(-d-n-1)) \to \mathcal{O}_Z
$ is a perfect pairing of finitely generated free $\mathcal{O}_Z$-modules.
\end{itemize}
\end{CPS}

\begin{rem}
As the first statement indicates, $H^0(\mathbf{P}_Z^n, \mathcal{O}_{\mathbf{P}_Z^n}(d))$ can be interpreted as degree $d$ homogenous polynomials.  So in light of the last statement, 
$$\Gamma_*(\mathcal{O}_{\mathbf{P}_Z^n})^\vee = \bigoplus_d H^n(\mathbf{P}_Z^n, \mathcal{O}_{\mathbf{P}_Z^n}(-d-n-1))$$ should be thought of
as the coalgebra  dual to $\Gamma_*(\mathcal{O}_{\mathbf{P}_Z^n}).$
\end{rem}

\subsection{Relatively Flat Sheaves}
We have here some standard results on relatively flat sheaves.    These are at the core of many constructions
in the theory of moduli schemes.

The notion of relative flatness is largely motivated by interest in studying 
subsheaves $\mathcal{A}$ of a sheaf $\mathcal{B}.$  For $\mathcal{A}$ to remain a subsheaf
of $\mathcal{B}$ after base change, the map $\mathcal{A} \to \mathcal{B}$ must be a {\bf universal inclusion} (also called a ``universal injection'').

We find that the relationship between ``subobject'' and ``universal inclusion'' is made clear by considering when the subobject under consideration is or isn't a sheaf:
If the inclusion is not universal, then the subobject of $\mathcal{B}$ defined as the image of $\mathcal{A}$
is not a sheaf.  Conversely, if this subobject is a sheaf, then the inclusion is universal and the sheaf in question
is $\mathcal{A}.$

In general, it is difficult to recognize a universal inclusion.  However, if the cokernel of the inclusion
is relatively flat, then the map is automatically a universal inclusion. These give a class of universal inclusions we call
 {\bf cokernel-flat.}  If the ambient sheaf is the structure sheaf of a $Z$-scheme, and thus the subsheaf is an 
 ideal, then all universal inclusions are cokernel-flat. Otherwise, one must ``work'' to know if a given map is 
 a universal inclusion.

\begin{prop}
A sheaf $\mathcal{F}$ on a projective $Z$-scheme $\pi \colon B \to Z$ is {\bf relatively flat} if and only if any of the following equivalent conditions hold:
\begin{itemize}
\item for all $b \in B$, the stalk $\mathcal{F}_b$ is a flat $\mathcal{O}_{Z, \pi(b)}$-module;
\item for any affine subsets $U \subseteq B$  and $V \subseteq Z$ such that $\pi(U) \subseteq V$, we have $\mathcal{F}(U)$
is a flat $\mathcal{O}_Z(U)$-module;
\item $\Gamma_{\geq m}(\mathcal{F})$ is $Z$-flat for some $m$.
\end{itemize}
\end{prop}
\begin{proof}
Omitted
\end{proof}

\begin{prop}
If 
$$
0 \to \mathcal{F} \to \mathcal{G} \to \mathcal{H} \to 0
$$
is an exact sequence of quasi-coherent sheaves on a $Z$-scheme $B$, and $\mathcal{H}$ and either $\mathcal{G}$ or 
$\mathcal{F}$ are relatively flat, then all three sheaves are.
\end{prop}
\begin{proof}
Omitted.
\end{proof}

\begin{prop}
If $\mathcal{I}$ is a quasi-coherent ideal sheaf on a $Z$-scheme $B,$ then $\mathcal{I} \to \mathcal{O}_B$
is a universal inclusion if and only if $\mathcal{O}_B/\mathcal{I}$ is relatively flat.
\end{prop}
\begin{proof}
Omitted.
\end{proof}

\begin{SCT}(\cite{EGA_III_2}  see also \cite[Theorem III.12.8]{hartshorne-1977})
Let $\mathcal{F}$ be a coherent $Z$-flat sheaf on a projective  $Z$-scheme $B$. The function
$$
h^i(q, \mathcal{F}) = \dim_{\mathbf{k}(q)} H^i(B_q, \mathcal{F}|_q)
$$
is upper semicontinuous.
\end{SCT}

\begin{CBC}
(\cite{EGA_III_2}  see also \cite[Theorem III.12.11]{hartshorne-1977})
Let $\mathcal{F}$ be a coherent $Z$-flat sheaf on a projective $Z$-scheme $\pi \colon B \to Z$.  
For $q \in Z$, if 
$$
R^i\pi_*(\mathcal{F}) \otimes_{\mathcal{O}_Z} \mathbf{k}(q) \to H^i(B_q, \mathcal{F}|_q)
$$
is surjective, then it isomorphism.  In this case, 
it is an isomorphism for all $q'$
in a open set about $q$,
and the following statements are equivalent:
\begin{itemize}
\item $R^i\pi_*\mathcal{F}$ is flat at $q;$ 
\item The restriction map $$
R^{i-1}\pi_*(\mathcal{F}) \otimes_{\mathcal{O}_Z} \mathbf{k}(q) \to 
H^{i-1}(B|_q, \mathcal{F}|_q)
$$  
is surjective. 
\end{itemize}
\end{CBC}

\begin{rem}
When $R^i\pi_*(\mathcal{F}) \otimes_{\mathcal{O}_Z} \mathbf{k}(q) = H^i(B_q, \mathcal{F}|_q)$ we say 
{\it cohomology commutes with base change in degree $i$}.  In this case,
statements about $R^i\pi_*(\mathcal{F})$ are often reduced to statements about $H^i(B_q, \mathcal{F}|_q)$
(via Nakayama's lemma).
\end{rem}

The vanishing of the first cohomology of a sheaf on a single fiber has significant implications.
\begin{cor}
\label{cor:cbc}
Let $\mathcal{F}$ be a coherent $Z$-flat sheaf on a projective $Z$-scheme $\pi \colon B \to Z.$
If for $q \in Z$ we have $H^1(B_q, \mathcal{F}|_q) = 0,$ then for all $q'$ in a neighborhood of $q$ the sheaf
$\pi_*\mathcal{F}$ is flat at $q'$ and 
$\pi_*\mathcal{F}|_{q'} = H^0(B_{q'}, \mathcal{F}|_{q'}).$  
\end{cor}  
\begin{proof}
The Semicontinuity theorem implies
that for all $q'$ in a neighborhood $U$ of $q$ we have $H^1(B_{q'}, \mathcal{F}|_{q'}) = 0,$
and so Cohomology and Base Change for $i=1$ gives
$$R^1\pi_*\mathcal{F}|_{q'} = H^1(B_{q'}, \mathcal{F}|_{q'}) = 0$$ 
in $U.$
In particular 
$R^1\pi_*\mathcal{F}$ is flat on $U.$ 
So again with $i=1$, Cohomology and Base Change gives $\pi_*\mathcal{F}|_{q'} = H^0(B_{q'}, \mathcal{F}|_{q'})$ for all $q'$ in $U$.
Since $H^{-1}(B_{q'}, \mathcal{F}|_{q'}) = 0$ the restriction map is surjective, so 
Cohomology and Base Change for $i=0$ implies $\pi_*\mathcal{F}$ is flat on $U.$ 
\end{proof}

Mumford's notion of {\bf regularity} of a sheaf leads to a practical means of knowing when one can apply the Cohomology and Base Change theorem.  In the case of quasi-coherent sheaves of ideals, this concept along with the {\bf Gotzmann regularity theorem} 
give powerful tools.

\begin{defn}(\cite{mumford-curves})
Let $\mathbf{k}$ be a field.  A coherent sheaf $\mathcal{F}$ over $\mathbf{P}_{\mathbf{k}}^n$ is said to be 
$m$-regular if 
$$
H^i(\mathbf{P}^n_\mathbf{k}, \mathcal{F}(m-i)) = 0
$$
for each $i > 0.$
\end{defn}

\begin{rem}
If one puts $H^i(\mathbf{P}^n_\mathbf{k}, \mathcal{F}(j))$ at the location $(j,i)$ in the plane, then the
non-zero locations lie either on the $x$-axis, or below the line $x+ y = m.$  Also note, that if $\mathcal{F}$
is extended to a $Z$-flat family on $\mathbf{P}^n_Z,$ for some $Z,$  then the Cohomology and Base
Change theorem can be applied to $\mathcal{F}(d)$ for $d \geq m$ as in Corollary \ref{cor:cbc}.
\end{rem}

\begin{CM}(\cite{mumford-curves})
Let $\mathcal{F}$ be an $m$-regular coherent sheaf on $\mathbf{P}^n_\mathbf{k}.$
Then
\begin{itemize}
\item $\Gamma_{\geq m}(\mathcal{F})$ is generated in degree $m$ as a $\Gamma_*(\mathcal{O}_{\mathbf{P}^n_\mathbf{k}})$-module, 
\item $H^i(\mathbf{P}^n_\mathbf{k}, \mathcal{F}(d)) = 0$ whenever $d \geq m-i,$ and
\item each $\mathcal{F}(d)$ for $d \geq m$ is generated by its global sections.  
\end{itemize} 
\end{CM}

\begin{cor}
Let $Z$ be an affine Noetherian scheme.
If $\mathcal{F}$ is a $Z$-flat coherent sheaf on $\mathbf{P}_Z^n$ and $\mathcal{F}|_q$
is $m$-regular for all $q$ in $Z$, then 
\begin{itemize}
\item $\Gamma_{\geq m}(\mathcal{F})$ is $Z$-flat and  generated in degree $m$ as a $\Gamma_*(\mathcal{O}_{\mathbf{P}^n_Z})$-module, 
\item $R^i\pi_*( \mathcal{F}(d)) = 0$ whenever $d \geq m-i,$ and
\item each $\mathcal{F}(d)$ for $d \geq m$ is generated by its global sections.  
\end{itemize} 
\end{cor}
\begin{proof}
The second statement follows from Cohomology and Base change.  

The last can be checked considering 
$b \in \mathbf{P}_Z^n.$  Write $q = \pi(b).$  
Cohomology and Base change gives the surjection $\Gamma(\mathbf{P}^n_Z, \mathcal{F}(d)) \to \Gamma(\mathbf{P}^n_{\mathbf{k}(q)}, \mathcal{F}(d)|_q),$ and the theorem of Castelnuovo and Mumford gives the surjection
$$
\Gamma(\mathbf{P}^n_{\mathbf{k}(q)}, \mathcal{F}(d)|_q) \otimes_{\mathcal{O}_{\mathbf{P}^n_{\mathbf{k}(q)}} }
\mathbf{k}(b)
\to (\mathcal{F}(d)|_q)|_b.$$ The $\mathbf{k}(b)$-vector spaces 
$(\mathcal{F}(d)|_q)|_b$ and $\mathcal{F}(d)|_b$ equal, so we get a surjection 
$\Gamma(\mathbf{P}^n_Z, \mathcal{F}(d)) \otimes_{\mathcal{O}_{\mathbf{P}^n_{Z} }} 
\mathbf{k}(b)
\to \mathcal{F}(d)|_b,$ and can apply
Nakayama's lemma.  

The first statement requires consideration of the sheaves
 $\mathcal{O}(d),$  $\mathcal{F}(m)$ and   $\mathcal{F}(d+m).$
 Cohomology commutes with base change for all these sheaves, so we can consider 
 the question on the fiber.  The last fact we need is that
 $ (\pi_*\mathcal{O}(d) \otimes_{\mathcal{O}_Z} \pi_*\mathcal{F}(m))|_q \to 
 (\pi_*\mathcal{F}(d+m))|_q$ factors through the epimorphism 
 $$
 (\pi_*\mathcal{O}(d) \otimes_{\mathcal{O}_Z} \pi_*\mathcal{F}(m))|_q \to 
 \pi_*\mathcal{O}(d)|_q \otimes_{\mathbf{k}(q)} \pi_*\mathcal{F}(m)|_q.
 $$
We can now appeal to the theorem of Castelnuovo and Mumford and apply Nakayama's lemma.
 \end{proof}

\begin{GR}(\cite{gotzmann-1978} see also \cite[Theorem 4.3.2]{bruns-herzog})
If $L$ is a closed subscheme of $\mathbf{P}^n_\mathbf{k}$ with Hilbert polynomial $p_{\mathcal{O}_L}(t),$ then
there is a unique expansion
$$
p_{\mathcal{O}_L}(t) = { t+a_1 \choose a_1} + {t +a_2-1\choose a_2} + \dotsm + {t + a_m - (m-1)\choose a_m}
$$  
for weakly decreasing integers $a_1 \geq a_2 \geq \dotsm \geq a_m.$  Furthermore,
for $m$ in the above expansion
 $\mathcal{I}_L,$ is $m$-regular.
The integer $m$ is called the {\bf Gotzmann number} of $p_{\mathcal{O}_L}.$
\end{GR}

\begin{rem}
It is interesting to note that the Gotzmann number depends only on the polynomial.  Not even 
on the dimension of the ambient projective space.
\end{rem}

\begin{cor}
Consider a  $B$-flat closed subscheme $L$ of $\mathbf{P}^n_B$ 
with Hilbert polynomial $p_{\mathcal{O}_L}$ and Gotzmann number $m.$  The sheaves
$\pi_*{\mathcal{I}_L}$ and $\pi_*\mathcal{O}_L$ are $m$-regular, 
and
$$
0 \to \pi_*{\mathcal{I}_L(d)} \to \pi_*\mathcal{O}_{\mathbf{P}^n_B}(d) \to \pi_*\mathcal{O}_L(d) \to 0
$$ 
is an exact sequence of $Z$-flat sheaves. 
\end{cor}
\begin{proof}
$\mathcal{O}_L$ is $B$-flat, so after pull back to $\mathbf{P}^n_{\mathbf{k}(q)}$
the sequence 
$$
0 \to \mathcal{I}_L|_q \to \mathcal{O}_{\mathbf{P}^n_\mathbf{k}(q)} \to \mathcal{O}_L|_q \to 0
$$
is exact.  Here Gotzmann regularity implies $H^i( \mathbf{P}^n_{\mathbf{k}(q)}, \mathcal{I}_L|_q(d)) = 0$
for all $i >0.$
These groups can be computed by the same \v{C}ech-complex as the $R^i\pi_*( \mathcal{I}_L(d)  \otimes_{\mathcal{O}_{\mathbf{P}^n_B}} \mathbf{k}(q))$'s, and are thus the same. 
So we may apply Cohomology and Base Change to conclude 
 $R^i\pi_*( \mathcal{I}_L(d)) = 0$ for $i >0.$  Consequently, both $\pi_* \mathcal{I}_L(d)$
  and $\pi_*\mathcal{O}_L(d)$
are flat.
\end{proof}

\begin{notation}
When we have a projective $Z$-scheme $B$ and a sheaf $\mathcal{F}$ on $B,$
we will say some version of the statement 
\begin{center}
``The Hilbert polynomial of $\mathcal{F}$ is independent of $Z.$''
\end{center}
to indicate that 
there is a fixed polynomial that  equals the Hilbert polynomial of $\mathcal{F}|_q$ regardless of the choice of point $q \in Z$.
\end{notation}

Later (Lemma \ref{lem:reduced-hilbert-constant-flat}), we will need to generalize the following theorem to reduced schemes.

\begin{HPRF}(\cite[proof of Theorem III.9.9]{hartshorne-1977})
If $Z$ is an integral Noetherian scheme.  Let $\mathcal{F}$ be a coherent sheaf on  a projective $Z$-scheme
$B$.  Then $\mathcal{F}$ is $Z$-flat if and only if Hilbert polynomial of $\mathcal{F}|_q$ is independent of $q \in Z$.
\end{HPRF}

\subsection{Representability of certain functors}
We will use certain schemes in a way that makes it convenient to think of them in terms of the functors they represent.  Specifically, {\bf flattening stratifications}, {\bf Quot schemes}, {\bf Hilbert schemes}, and {\bf scheme theoretic images}.

\begin{Flattening}(\cite{grothendieck-hilbert})
If $B \to Z$ is projective and  $\mathcal{F}$ is a coherent $\mathcal{O}_B$
module, then there  is $Z$-scheme  $Z^\mathcal{F}_{\text{flat}} \to Z$ which represents the functor 
$$
Z' \mapsto \{ h \colon Z' \to Z \ | \ h^*\mathcal{F} \text{ is $Z'$-flat}\}.
$$
Furthermore, $Z^\mathcal{F}_{\text{flat}}$ is a disjoint union of locally closed subschemes of $Z$
called {\bf strata}, one of which is topologically open and dense in $Z.$
\end{Flattening}

\begin{notation}
If $\mathcal{F}$ is a sheaf over a projective $Z$-scheme and $p_\mathcal{F}= p_{\mathcal{F}}(t)$ 
is a polynomial in $t$, we write
$Z_{p_{\mathcal{F}}}$ for the disjoint union of locally closed subschemes over which $\mathcal{F}$ is flat and has Hilbert polynomial $p_{\mathcal{F}}$.
 This is potentially confusing since suggests 
that $p_\mathcal{F}$ depends on $\mathcal{F}.$  However, this notation should simply indicate that we 
are introducing a  polynomial  $p_\mathcal{F}$ that we wish to associate with the sheaf $\mathcal{F}.$
\end{notation}

\begin{Quot}  (\cite{grothendieck-hilbert}) Given a coherent sheaf $\mathcal{F}$ over a projective $Z$-scheme $B$ the functor
$$
Z' \mapsto   \{ \text{$Z'$-flat quotients $\mathcal{G}$ of $\pi_B^*\mathcal{F}$ on $Z' \times_Z B$} \} 
$$
is representable by a projective $B$-scheme 
$\operatorname{Quot}(\mathcal{F}/B/Z).$
\end{Quot}

\begin{rem}
From this we have the Hilbert scheme which is $\operatorname{Hilb}(B/Z) = 
\operatorname{Quot}(\mathcal{O}_B/B/Z).$
\end{rem}

\begin{STI}(\cite[Tag 01R5]{stacks-project})
Given a morphism of schemes $\phi \colon V \to W.$ There exists a closed subscheme $\overline{\phi(V)} \subseteq W$ 
called the {\bf scheme theoretic image} such that $\phi$ factors through  $\overline{\phi(V)}$ and $\overline{\phi(V)}$
is initial among such closed subschemes of $W.$
\end{STI}

\section{The Moduli of Quasi-Splines Sheaves}
\label{section:moduli-quasi-spline-sheaves}

In this section, we construct in Theorem \ref{thm:cfqs} the moduli of cokernel-flat families of quasi-spline sheaves ${CFQS}.$
The functor
represented by ${CFQS}$ is
$$
\mathcal{CFQS}(Z) = \{ \mathcal{S} \in \mathcal{QS}(Z) \ | \ \text{$\mathcal{G}=\operatorname{cok}(\mathcal{S} \to \mathcal{O}_{Z \times_T Y}^s)$ is $Z-$flat}  \}.
$$
Where $\mathcal{QS}$ is the functor of families of quasi-spline schemes from the introduction.  

When $Y$ is $T$-flat  we have Theorem \ref{thm:main} which states
the the existence of the scheme $QS = CFQS$ representing $\mathcal{QS}.$  This is based on
Lemma \ref{lem:flat-cok} makes the observation that a quasi-spline sheaf $\mathcal{S}$ over projective, flat $Z$-scheme $B$
is a $Z$-family if and only if the cokernel $\mathcal{G}$ of the inclusion $\mathcal{S} \to \mathcal{O}_B^s$
is $Z$-flat.

\begin{defn}
A {\bf sheaf of quasi-splines} over a scheme $B$ is a quasi-coherent $\mathcal{O}_B$-subalgebra of 
$\mathcal{O}_B^s.$
\end{defn}

\begin{defn}
A {\bf $Z$-family of quasi-spline sheaves over a $Z$-scheme $B$} 
to be
\begin{itemize}
\item a sheaf of quasi-splines $\mathcal{S}$
over a  $B$ such that 
\item for any morphism $f\colon Z' \to Z$, the pullback 
$\pi_B^*\mathcal{S}$
is a sheaf of quasi-splines over $Z' \times_Z B.$
\end{itemize}
\end{defn}

\begin{defn}
We say that a $Z$-family $\mathcal{S}$ of quasi-spline sheaves over $B$ is 
{\bf cokernel-flat} if the sheaf $\mathcal{G}$ is the exact sequence
$$
0 \to \mathcal{S} \to \mathcal{O}^s_B \to \mathcal{G} \to 0
$$
is $Z$-flat.
\end{defn}

\begin{lem}
\label{lem:zero-morphism}
Let $\phi \colon \mathcal{F} \to \mathcal{G}$ be a morphism of coherent sheaves over a projective $Z$-scheme $B.$
If $\mathcal{G}$ is $Z$-flat, then the functor $Z' \mapsto \{ h \in \operatorname{Mor}_T(Z', Z) \ | \ 
h^*\phi = 0\}$
is representable by a closed subscheme $\mathbf{V}(\phi) \subseteq Z.$ 
\end{lem}
\begin{proof}
It suffices to work locally on $Z$ and assume that $B \subseteq \mathbf{P}^n_Z.$  
Provided $d$ is sufficiently large, $\mathcal{F}(d)$ is generated by global sections and $\Gamma(B, \mathcal{G}(d))$
is a flat $\mathcal{O}_Z$-module.  Consider the image 
under $\phi(d) \colon \Gamma(B, \mathcal{F}(d)) \to \Gamma(B, \mathcal{G}(d))$
of generators $\{f_i\}_i \subseteq \Gamma(B, \mathcal{F}(d)).$  Since $Z$ is local,
$\Gamma(B, \mathcal{G}(d))$ is free and we can choose a basis $\{g_j\}_j.$
For each $f_i$ we have an expansion 
$$
\phi(f_i) = \sum_j c_{ij} g_j.
$$
The condition that $\phi = 0$ is the same as $c_{ij} = 0$ for all $ij.$
So we set $\mathbf{V}(\phi) = \mathbf{V}(\{c_{ij}\}).$  After any base change, $\mathcal{F}(d)$
is still generated by the $f_i$'s and the cohomology and base change theorem implies that 
the $g_j$'s remain linearly independent.  So the vanishing of $\phi$ is exactly the condition that the
$c_{ij}$'s vanish.
\end{proof}

\begin{defn}
\label{defn:cm}
Fix a scheme $B$ and consider a quasi-coherent subsheaf $\mathcal{S}$ of $\mathcal{O}_B^s.$
Write $\gamma \colon \mathcal{O}_B^s \to \mathcal{G}$ for the cokernel of the inclusion $\iota \colon \mathcal{S} \to \mathcal{O}_B^s.$
 Write $\delta \colon \mathcal{O}_B \to \mathcal{O}_B^s$
for the diagonal inclusion and $\mu \colon \mathcal{O}_B^s \otimes_{\mathcal{O}_B} \mathcal{O}_B
\to \mathcal{O}_B^s$ for entry-wise multiplication.  We define 
\begin{itemize}
\item $ \kappa \colon \mathcal{O}_B \to \mathcal{G}$ to be the composition $\gamma \circ \delta,$ and
\item $m \colon \mathcal{S} \otimes_\mathcal{O} \mathcal{S} \to \mathcal{G}$ to be 
$\gamma \circ \mu \circ (\iota \otimes \iota).$
\end{itemize}
\end{defn}

\begin{lem}
\label{lem:cm}
A quasi-coherent subsheaf $\mathcal{S}$ of $\mathcal{O}_B^s$ is a quasi-spline sheaf if and only if
$$
\text{ $ \kappa = 0$ and $m = 0.$ }
$$
for the maps $ \kappa$ and $m$ of Definition \ref{defn:cm}.
\end{lem}
\begin{proof}
$ \kappa =0$  and $m=0$ if and only if they factor through the kernel of $\gamma.$  So
both  $\delta$ and $\mu \circ (\iota \otimes \iota)$ factor though $\mathcal{S}.$  This 
 means precisely that $\mathcal{O}_B \subseteq \mathcal{S}$
and $\mathcal{S}$ is closed under entry-wise multiplication in $\mathcal{O}_B^s.$
\end{proof}

\begin{defn}
Consider the Quot scheme $Q = \operatorname{Quot}(\mathcal{O}_Y^s/Y/T)$
 and the map 
$$
\phi  = \kappa \oplus m \colon \mathcal{O}_{Q \times_T Y} \oplus (\mathcal{S} \otimes_{\mathcal{O}_{Q \times_T Y}} \mathcal{S}) \to \mathcal{G} \oplus \mathcal{G}.
$$
where $ \kappa$ and $m$ are the maps from Definition \ref{defn:cm} for the universal kernel $\mathcal{S} \subseteq \mathcal{O}_{Q \times_T Y}^s.$
We set
$$
{CFQS} = \mathbf{V}(\phi) \subseteq Q
$$
as in Lemma \ref{lem:zero-morphism}. 
\end{defn}

\begin{thm}
\label{thm:cfqs}
The functor $\mathcal{CFQS}$ is represented by  ${CFQS}.$ 
\end{thm}
\begin{proof}
For any such family we know by Lemma \ref{lem:flat-cok} that the cokernel $\mathcal{G}$ of the inclusion
$\mathcal{S} \to \mathcal{O}^s_{Z\times_TY}$ is $Z$-flat.  Furthermore, the quasi-spline sheaf
is determined by the map $\mathcal{O}^s_{Z\times_TY} \to \mathcal{G}.$
This means that there is a natural transformation 
from this functor into the Quot scheme.

The identification
$$\mathcal{S} \otimes_{\mathcal{O}_{Q \times_T Y}} \mathcal{S} \otimes_{\mathcal{O}_{Q \times Y}} \mathcal{O}_{Z \times_TY} 
\cong 
(\mathcal{S} \otimes_{\mathcal{O}_{Q \times_T Y}} \mathcal{O}_{Z \times_T Y})
\otimes_{\mathcal{O}_{Z \times_T Y}} (\mathcal{S} \otimes_{\mathcal{O}_{Q \times_T Y}} \mathcal{O}_{Z \times_T Y})
$$
shows that the $ \kappa$ and $m$ maps of Definition \ref{defn:cm} over $Q \times_T Y$ pull back to the $ \kappa$ and $m$ maps over $Z \times T.$
For any quasi-spline sheaf  over $Z \times T$ these maps vanish by Lemma \ref{lem:cm},
so the morphism factors through ${CFQS}$ by Lemma \ref{lem:zero-morphism}. 

On the other hand, the universal kernel $\mathcal{S}$ restricted to ${CFQS}$
is a quasi-spline sheaf, again by Lemma \ref{lem:cm}.  Lemma \ref{lem:flat-cok} guarantees this sheaf is a ${CFQS}$-family.
 Consequently, points in ${CFQS}(Z)$ produce distinct $Z$-families of quasi-spline sheaves over
$Z \times_T Y$, and so the natural transformation is a bijection.
\end{proof}

\begin{lem}
\label{lem:flat-cok}
\label{lem:yfcf}
Given a flat, projective $Z$-scheme $B$, a quasi-spline sheaf $\mathcal{S}$ is a 
$Z$-family over $B$ if and only if the cokernel $\mathcal{G}$
of the inclusion $\mathcal{S} \to \mathcal{O}_B^s$ is $Z$-flat. Thus 
if $Y$ is flat over $T$, then 
$\mathcal{QS}= 
\mathcal{CFQS}.$
\end{lem}
\begin{proof}
$B$ is flat, thus so is $\mathcal{O}_B^s.$   Consider the exact sequence 
at a point $b \in B$:
$$
0 \to \mathcal{S}_b \to \mathcal{O}_{B,b}^s \to \mathcal{G}_b \to 0.
$$
Denote the image of $b$ in $Z$ by $q.$ Tensoring with the sheaf $\mathbf{k}(q)$
we get the $\operatorname{Tor}$ exact sequence
$$
0 \to \operatorname{Tor}^{\mathcal{O}_Z}_1(\mathcal{G}_b, \mathbf{k}(q)) \to 
\mathcal{S}_b \otimes_{\mathcal{O}_Z} \mathbf{k}(q) \to 
\mathcal{O}_{B,b}^s \otimes_{\mathcal{O}_Z} \mathbf{k}(q) \to 
\mathcal{G}_b \otimes_{\mathcal{O}_Z} \mathbf{k}(q) \to 
0.
$$
So we see that $\mathcal{S} \to \mathcal{O}_B^s$ is a universal inclusion
if and only if  $\operatorname{Tor}^{\mathcal{O}_Z}_1(\mathcal{G}_b, \mathbf{k}(q)) = 0$
for all $b \in B,$ i.e. $\mathcal{G}$ is $Z$-flat.
\end{proof}

\begin{thm}
\label{thm:main}
If $Y$ is $T$-flat, $QS = CFQS$ represents $\mathcal{QS}.$
\end{thm}
\begin{proof}
Combine Lemma \ref{lem:yfcf}.
and Theorem \ref{thm:cfqs}.
\end{proof}

\begin{rem}
If $Y$ is not $T$-flat, there is no chance that $QS = CFQS$ (in the way presented here).  For instance, 
one can take $\mathcal{S}=$ the diagonal copy of $\mathcal{O}_Y \subseteq \mathcal{O}_Y^s.$
The inclusion is universal and it spits $\mathcal{O}_Y^s$ into a direct sum $\mathcal{O}_Y \oplus \mathcal{O}_Y^{s-1}.$
So the cokernel is isomorphic to $\mathcal{O}_Y^{s-1},$ and not $T$-flat.
\end{rem}

\section{The Moduli of Ideal Difference-Conditions}
\label{section:ideal-difference}
We begin by constructing the moduli of ideal difference-conditions $C$ in Theorem \ref{thm:condition-moduli} as a flattening stratification
of a certain sheaf over 
$\operatorname{Hilb}(Y/T)^{s \choose 2}.$  To produce our ``compactification'' of this scheme,
we show in Proposition \ref{prop:C-subscheme} it is a subscheme of ${CFQS} \times_T \operatorname{Hilb}(Y/T)^{s \choose 2},$ and define the compactification to be the scheme theoretic image $\overline{C}$ of the inclusion.
Finally, we argue via Proposition \ref{prop:C-correct} that this compactification is the ``correct'' one.

\subsection{The Moduli of Ideal Difference-Conditions}
For ideal difference-conditions defined by a collection of ideals $(\mathfrak{I}_{jk})_{jk}$
over a $Z$-scheme $B,$ we consider 
the morphism 
\begin{equation}
\label{equation:ideal-difference}
\Delta \colon \mathcal{O}_{B}^s \to \bigoplus_{jk} \mathcal{O}_{B} / \mathfrak{I}_{jk}.
\end{equation}
which sends $(g_1, \dotsc, g_s) \mapsto (g_{j}-g_k + \mathfrak{I}_{jk})_{jk}.$  Not all collections of ideals are well behaved.

\begin{defn}
Recall the notion of  a {\bf $Z$-family} of ideal difference-conditions: Under  base change along any morphism $Z'  \to Z$  the sequence
$$ 
0 \to \mathcal{S}_\mathfrak{I} \to \mathcal{O}_{B}^s \to \bigoplus_{jk} \mathcal{O}_{B} / \mathfrak{I}_{jk}.
$$
remains exact, and for each ${jk}$ the sequence
$$
0 \to \mathfrak{I}_{jk} \to  \mathcal{O}_{B} \to \mathcal{O}_{B} / \mathfrak{I}_{jk} \to 0
$$
also remains exact.  This means sheaves $(\mathcal{S}_\mathfrak{I}, (\mathfrak{I}_{jk})_{jk})$ ``remain themselves.''
after  such a change of base.
\end{defn}

\begin{lem}
\label{lem:characterize-diff-con}
A collection of ideals $(\mathfrak{I}_{jk})_{jk}$ over a $Z$-scheme $B$ define a $Z$-family of ideal difference-conditions if and only if 
both 
$\bigoplus_{jk} \mathcal{O}_{B} / \mathfrak{I}_{jk}$
and the cokernel of 
$$
\Delta \colon \mathcal{O}_{B}^s \to \bigoplus_{jk} \mathcal{O}_{B} / \mathfrak{I}_{jk}
$$ are $Z$-flat.  In this case, the cokernel $\mathcal{G}$ of 
$\mathcal{S}_\mathfrak{I} \to \mathcal{O}_{B}^s$ is automatically $Z$-flat.
\end{lem}
\begin{proof}
For ideals, universal inclusions are equivalent to flatness of their cokernels, so
a $Z$-family requires $
\mathcal{O}_{Z \times_T Y} / \mathfrak{I}_{jk} 
$
is $Z$-flat for all $jk.$  Given this, the additional required conditions 
reveal themselves after considering the  two standard exact sequences associated to the morphism in Equation (\ref{equation:ideal-difference}):
\begin{equation}
\label{eq:spline-inclusion}
0 \to \mathcal{S}_{\mathfrak{I}} \to \mathcal{O}_{B}^s  \to \mathcal{G} \to 0
\end{equation}
and 
\begin{equation}
\label{eq:cokernel-quotient}
0 \to \mathcal{G} \to \bigoplus_{jk} \mathcal{O}_{B} / \mathfrak{I}_{jk} \to \mathcal{H} \to 0.
\end{equation}
If either of these exact sequences fail to be exact after base change, $\mathcal{S}_{\mathfrak{I}}$
will no longer be the kernel of Equation (\ref{equation:ideal-difference}).

Since  $\bigoplus_{jk} \mathcal{O}_{Z \times_T Y} / \mathfrak{I}_{jk}$ must be $Z$-flat, universal exactness of the second  standard sequence
is equivalent to  the  $Z$-flatness of $\mathcal{H}.$  This implies the $Z$-flatness of $\mathcal{G},$ and thus the
exactness of the first standard sequence.
\end{proof}

\begin{defn}
Given a projective scheme $Y/T,$ denote the structure sheaf of 
product $\operatorname{Hilb}(Y/T)^{s \choose 2}$ 
of Hilbert Schemes by $\mathcal{O}.$  Over this product we have the morphism
$$
\Delta \colon \mathcal{O}^s \to \bigoplus_{jk} \mathcal{O}/\mathfrak{I}_{jk}.
$$
Denote the cokernel by $\mathcal{H}$ and we define 
{\bf the moduli of ideal difference-conditions} ${C} = $
the universal flattening stratification for $\mathcal{H}.$
\end{defn}

\begin{thm}
\label{thm:condition-moduli}
The functor
$$
\mathcal{C}(Z) = \{\text{$Z$-families of  ideal difference conditions on $Z \times_T Y$}\}
$$
is representable by ${C}.$
\end{thm}
\begin{proof}
The definition of the Hilbert Scheme, the universal flattening stratification, and Lemma \ref{lem:characterize-diff-con} give the result.
\end{proof}

\subsection{Compatible Pairs and Compactification of the Moduli of Ideal Difference-Conditions}
We now have a construction of the moduli scheme of ideal difference conditions.  However, to construct a satisfying ``compactification,''  we present it in a slightly different way.  This involves the observation that the assignment
$$
(\mathfrak{I}_{jk})_{jk} \mapsto \mathcal{S}_\mathfrak{I}
$$
defines a morphism ${C} \to {CFQS}.$  It will turn out that the resulting morphism
$$
{C} \to {CFQS} \times_T \operatorname{Hilb}(Y/T)^{s \choose 2}
$$
is an inclusion of ${C}$ as a locally closed subscheme, and its scheme theoretic closure $\overline{{C}}$
it the ``correct'' compactification.  The correctness of $\overline{{C}}$ is based on the existence of a universal family of 
compatible pairs (Definition \ref{defn:compatible-pairs}) and its universality (Proposition \ref{prop:C-correct}).

\begin{defn}
\label{defn:compatible-pairs}
A pair of a quasi-spline sheaf $\mathcal{S}$ and a ${s \choose 2}$-tuple of ideal sheaves $(\mathfrak{I}_{jk})_{jk}$
over a scheme $B$ is called a {\bf compatible pair}
if the composition 
$$
\mathcal{S} \to \mathcal{O}^s \to  \bigoplus_{jk} \mathcal{O}_{B} / \mathfrak{I}_{jk}
$$
is zero.
\end{defn}

\begin{defn}
Denote by $\mathcal{O}$ the structure sheaf of ${CFQS} \times_T \operatorname{Hilb}(Y/T)^{s \choose 2} \times_T Y.$  
Over ${CFQS} \times_T \operatorname{Hilb}(Y/T)^{s \choose 2} \times_T Y$ we have the universal pair $(\mathcal{S}, (\mathfrak{I}_{jk})_{jk})$
and the compatibility map 
 $$\psi \colon \mathcal{S} \to \bigoplus_{jk} \mathcal{O} / \mathfrak{I}_{jk}.$$

The sheaf
$\bigoplus_{jk} \mathcal{O} / \mathfrak{I}_{jk}$ is relatively flat over ${CFQS} \times_T \operatorname{Hilb}(Y/T)^{s \choose 2},$ so
Lemma \ref{lem:zero-morphism} produces the {\bf moduli of compatible pairs}
$$
P = \mathbf{V}(\psi) \subseteq {CFQS} \times_T \operatorname{Hilb}(Y/T)^{s \choose 2}.
$$
\end{defn}

There is a natural morphism $f \colon {C} \to P$ which sends 
$(\mathfrak{I}_{jk})_{jk} \mapsto (\mathcal{S}_{\mathfrak{I}}, (\mathfrak{I}_{jk})_{jk}).$

\begin{lem}
\label{lem:reduced-hilbert-constant-flat}
Let $Z$ be a locally Noetherian and  $\mathcal{F}$ be a coherent sheaf over a reduced projective $Z$-scheme $B.$  
Then $\mathcal{F}$ is $Z$-flat if and only if the Hilbert polynomial of $\mathcal{F}$ is locally independent of $Z$.
\end{lem}
\begin{proof}
The question is local on $Z$, so assume $Z$ is affine and thus has finitely many irreducible components. 
Write $Z_1, \dotsc,  Z_k$ for the irreducible components of $Z,$ and write $Z^\mathcal{F}_\text{flat}$ for the flattening 
stratification of $Z$ for $\mathcal{F}.$  Over each $Z_i$ the restriction $\mathcal{F}$ is flat  by Hartshorne III.9.9, so we get a morphism $\phi_i \colon Z_i \to Z^\mathcal{F}_\text{flat}.$
$\phi_i$ and $\phi_j$ agree on  $Z_i \cap Z_j.$ 
It is a quick check (the Chinese Remainder Theorem) to see that one can patch maps on a pair of closed subschemes to produce one on their
union if the maps agree on the intersection.
So these maps define  $\phi_{ij} \colon Z_i \cup Z_j \to Z^\mathcal{F}_\text{flat}.$  
If we take this as a base case, the same argument produces a map $\phi_{i_1 \dotsm i_\ell} \colon Z_{i_1} \cup \dotsm \cup Z_{i_\ell} \to Z^\mathcal{F}_\text{flat}$ from 
$\phi_{i_1 \dotsm i_{\ell-1}}$ and $\phi_{i_\ell}.$  There are only finitely many components, so we get a morphism $\phi \colon Z \to Z^\mathcal{F}_\text{flat}.$
This morphism is a section of the map $Z^\mathcal{F}_\text{flat} \to Z.$  Since $Z$ is reduced and $Z^\mathcal{F}_\text{flat}$ is a locally closed subscheme of $Z$, this map is an isomorphism.
\end{proof}

\begin{defn}
Denote the {\bf compactification by compatible pairs} $\overline{{C}}$ to be the scheme theoretic image of ${C}$ in $P.$  
\end{defn}

\begin{prop}
\label{prop:C-subscheme}
${C} \to \overline{{C}}$ is a locally closed immersion.
\end{prop}
\begin{proof}
Since $\overline{C}$ is a closed subset of $P$, we will show that $f \colon {C} \to P$ is a locally closed immersion.
The question is local on $T,$ so assume $Y \subseteq \mathbf{P}^n_T.$  This allows us to talk about Hilbert polynomials 
for $Z$-flat sheaves on schemes of the form $Z \times_T Y.$

First we establish that if we fix polynomials $p_\mathcal{G}$ and $p_{\mathfrak{I}_{jk}}$
for each $jk,$ the scheme
${C}_{p_\mathcal{G}, p_\mathfrak{I}}$
is a union of  connected components of ${C}$.  To do this, we 
must verify that the Hilbert polynomials of $\mathcal{G}$ and the $\mathfrak{I}_{jk}$'s
are locally independent of the base.  

The Hilbert polynomials 
$p_{\mathfrak{I}_{jk}}(t)$ for each $jk$ 
are locally independent of ${C}$  by virtue of the fact that the $\mathfrak{I}_{jk}$'s
pull back from $\operatorname{Hilb}(Y/T)^{s \choose 2}.$  For 
the Hilbert polynomial of $\mathcal{G},$ observe that we have the equation 
\begin{equation}
\label{eq:hilb-poly-lin-relat}
 p_\mathcal{G}(t) = {s \choose 2}p_{\mathcal{O}_Y}(t) - \sum_{jk}p_{\mathfrak{I}_{jk}}(t) - p_\mathcal{H}(t).
\end{equation}
where we continue to denote the cokernel of the morphism $\mathcal{G} \to \bigoplus_{jk} \mathcal{O}_{B} / \mathfrak{I}_{jk}$ by 
$\mathcal{H}.$  This equation shows that the local constancy of 
$p_\mathcal{H}(t)$ is equivalent to the local constancy of 
$p_\mathcal{G}(t)$ (provided the $p_{\mathfrak{I}_{jk}}(t)$'s
are locally independent of the base).  As the flattening stratification of $\mathcal{H}$, $p_\mathcal{H}(t)$ is locally independent of $C.$

Before considering $P,$ we establish the  map ${C}_{p_\mathcal{G}, p_\mathfrak{I}} \to \operatorname{Hilb}(Y/T)^{s \choose 2}$
is a locally closed immersion.  This is topological statement since we know that, as a flattening stratification, ${C}$ is the disjoint union of locally closed 
subschemes.  Consider the image $\Lambda$ of 
${C}_{p_\mathcal{G}, p_\mathfrak{I}}$ as a topological space.  Choose a connected component $\Lambda'$ of $\Lambda.$
The Hilbert polynomials of $\mathcal{H}$ and the ideals $\mathfrak{I}_{jk}$ are locally independent of $\Lambda',$ so there is an open subset $U$ 
of the closure $\overline{\Lambda'}$ in $\operatorname{Hilb}(Y/T)^{s \choose 2}$ on which all these polynomials are locally independent of the point in $U$. 
This is a constructible set containing $\Lambda'.$ 
Equipped with its reduced scheme structure, the local independence of the Hilbert polynomials of $\mathcal{H}$
and the $\mathfrak{J}_{jk}$'s imply 
by Lemma \ref{lem:reduced-hilbert-constant-flat} these sheaves are flat over $U.$
So it admits a section $U\to {C}_{p_\mathcal{G}, p_\mathfrak{I}}$ and we can conclude that
$U = \Lambda'.$  In other words, any connected component $\Lambda'$ of $\Lambda$ is the homeomorphic
image of a connected component of ${C}_{p_\mathcal{G}, p_\mathfrak{I}}.$  Thus 
${C}_{p_\mathcal{G}, p_\mathfrak{I}} \to \operatorname{Hilb}(Y/T)^{s \choose 2}$
is a locally closed immersion.

Finally, we consider the morphism ${C} \to P.$  Denote by ${CFQS}_{p_\mathcal{G}}$ the component of 
${CFQS}$ over which $\mathcal{G}$ has  Hilbert polynomial ${p_\mathcal{G}}.$  We see that 
${C}_{p_\mathcal{G}, p_\mathfrak{I}}$ is carried to 
${CFQS}_{p_\mathcal{G}} \times_T {C}_{p_\mathcal{G}, p_\mathfrak{I}}.$
This is a locally closed subscheme of ${CFQS} \times_T\operatorname{Hilb}(Y/T)^{s \choose 2},$
so $P \cap ({CFQS}_{p_\mathcal{G}} \times_T {C}_{p_\mathcal{G}, p_\mathfrak{I}})$
is a locally closed subscheme of $P.$
${C}_{p_\mathcal{G}, p_\mathfrak{I}}$ itself is identified with the open subset of 
$P \cap ({CFQS}_{p_\mathcal{G}} \times_T {C}_{p_\mathcal{G}, p_\mathfrak{I}})$
over which $\mathcal{S} \to \mathcal{S}_\mathfrak{I}$ is an isomorphism.   To be sure that this 
is an open set, apply $\Gamma_*$ to the map, and notice this set coincides with the points where 
the cokernel and kernel vanish.
Thus as a open subscheme  of a
locally closed subscheme of $P,$ it is locally closed. 
\end{proof}

We conclude with the universal property of $\overline{C}.$

\begin{prop}
\label{prop:C-correct}
Let ${C} \to H$ be a morphism to a scheme such that $H \times_T Y$
is equipped 
with an $H$-family of compatible pairs whose restriction to ${C} \times_T Y$ 
is the universal family
of ideal difference-conditions.
Assume $H$ equals the scheme theoretic image of ${C}$ in $H$.  Then 
the morphism $H \to P$ factors through $\overline{{C}}.$
\end{prop}
\begin{proof}
$H \times_P \overline{{C}}$ equals $H$ because it is a closed subscheme of $H$ containing the image of ${C}:$
$$
\begin{tikzcd}
 {C} \arrow{dr} \arrow{r} & H \times_P \overline{{C}} \arrow{r} \arrow{d} &\overline{{C}} \arrow{d}\\
\  & H \arrow{r} &P.
\end{tikzcd}
$$
\end{proof}

\subsection{Degeneracy Loci and Rank Strata}
The construction of the moduli space gives a space in which one can move without changing the Hilbert polynomial.  It seems likely that one would be interested in the the whole Hilbert series, not just the polynomial. 
To address this, we present a way in which one can stratify the moduli space by pieces on which the Hilbert series is unchanged using degeneracy loci.  This can be done provided cohomology and base change commute for $\mathcal{O}_B^s(d)$ and 
$\bigoplus_{jk}   (\mathcal{O}_B/\mathfrak{I}_{jk})(d)$ for all $d.$  This condition holds in the most important case of $B = \mathbf{P}^n$
as shown in Corollary \ref{cor:deg-loci-pn}.

\begin{defn}
Given a morphism  $\phi \colon \mathcal{F} \to \mathcal{G},$ 
of flat coherent sheaves on a scheme $Z,$
 we say  
$$\operatorname{rank}(\phi) \leq r $$
if the induced morphism 
$$
\wedge^r \phi \colon \wedge^r \mathcal{F} \to \wedge^r \mathcal{G}
$$
is zero.  This map is a  global section 
$$
\wedge^r \phi \in \Gamma(Z,  \ \wedge^r (\mathcal{G} \otimes_{\mathcal{O}_Z} \mathcal{F}^\vee) \ ).
$$
The sheaf $\wedge^r (\mathcal{G} \otimes_{\mathcal{O}_Z} \mathcal{F}^\vee)$ is locally free, so 
$\wedge^r \phi$ defines a {\bf scheme of zeros} $(\wedge^r \phi)_0.$
This scheme is called the  {\bf $r^\text{th}$ degeneracy locus of $\phi$}, and we will denote it
by $DL_r(\phi).$
\end{defn}

\begin{prop}
\label{prop:degeneracy-loci}
Let $B$ be a projective, flat $Z$-scheme.  Consider ${s \choose 2}$ quasi-coherent ideals
$\mathfrak{I}_{jk} \subseteq \mathcal{O}_B$ with $Z$-flat quotients
$\mathcal{O}_B/\mathfrak{I}_{jk}.$
If $H^1(B_q, \bigoplus_{jk} \mathcal{O}_{B_q}/(\mathfrak{I}_{jk})_q)$ and
$H^1(B_q,\mathcal{O}_{B_q}^s)$
vanish for all $q \in Z$, then the locally closed subset on which  $\pi_*\mathcal{S}(d)$
has rank $\rho$ is 
$$
DL_r(\Delta(d)) \setminus DL_{r-1}(\Delta(d))
$$ 
where 
$$
\Delta(d) \colon \pi_*{\mathcal{O}_B}^s(d) \to  \pi_* \bigoplus_{jk}   (\mathcal{O}_B/\mathfrak{I}_{jk})(d)
$$
and 
$$
r = \operatorname{rank}( \pi_*{\mathcal{O}_B^s(d)}) - \rho.
$$
\end{prop}
\begin{proof}
The vanishing of
$H^1(B_q, \bigoplus_{jk} \mathcal{O}_{B_q}/(\mathfrak{I}_{jk})_q)$ and
$H^1(B_q,\mathcal{O}_{B_q}^s)$
 plus the
Corollary \ref{cor:cbc}
 guarantee that 
$ \pi_*\bigoplus_{jk} (\mathcal{O}_B/\mathfrak{I}_{jk})(d)$ and 
$\pi_*{\mathcal{O}_B^s(d)}$ are locally free $\mathcal{O}_Z$-modules.

Locally, $\Delta(d)$ is a matrix, and its cokernel is flat over a scheme $Z' \to Z$
exactly when this matrix pulls back to a constant rank matrix over $Z'.$  This 
is the same a requiring  that $Z'$ maps into 
$
DL_r(\Delta(d)) \setminus DL_{r-1}(\Delta(d))
$
for some $r.$  

The formula relating $r$ and $\rho$ follows from the fact that if $\Delta(d)$
is constant rank, then $\mathcal{S}(d)$ is flat of the given rank.
\end{proof}

\begin{cor}
\label{cor:deg-loci-pn}
Set  $B=\mathbf{P}^n_Z$  and $\mathfrak{I}_{jk} = \mathcal{O}_B(-D_{jk})$ for relatively effective Cartier divisors $D_{jk}.$
If $n \neq 2$ or  $\deg(D_{jk}) < n-1$ for all $jk$, then 
the locus on which $\pi_*\mathcal{S}(d)$ has rank $\rho$ is  
$$
DL_r(\Delta(d)) \setminus DL_{r-1}(\Delta(d))
$$ 
where
$$
r = s{d + n \choose 2} - \rho.
$$
\end{cor}
\begin{proof}
The conditions on $n$ and/or the $D_{jk}$'s guarantee these sheaves have no first cohomology.  So we may apply Proposition \ref{prop:degeneracy-loci}.
\end{proof}

\begin{rem}
To connect this with the moduli space, one might begin with the product of Hilbert schemes of degree $d_{jk}$
hypersurfaces
$\prod_{jk} \mathbf{P}(\Gamma(\mathbf{P}^n, \mathcal{O}(d_{jk}))^\vee)$ equipped with the bundle
 $\mathcal{H} = \operatorname{cok} \Delta.$
Then $Z$ would be taken to be the flattening stratification of $\mathcal{H}.$  These degeneracy loci then give the stratification of the moduli space on which the Hilbert series, not just the Hilbert polynomials are unchanged.
\end{rem}

\section{The Moduli of Quasi-Splines}
\label{section:moduli-quasi-splines}
In this section, we prove for a $T$-flat closed subscheme $Y$ of a projective space bundle $\mathbf{P}(\mathcal{V}),$
the functor which picks out a $T$-family of quasi-spline sheaves $\mathcal{S}$ and 
a section of a $d^\text{th}$ twist of $\mathcal{S}$
is representable, provided $d$ is sufficiently large.  The bound we find for $d$ depends on the Hilbert polynomials of 
$\mathcal{O}_Y$ and $\mathcal{S}.$   Since $Y$ is $T$-flat, we have a scheme $QS$ representing $\mathcal{QS},$
and this scheme equals $CFQS.$

The crucial thing we need is  a bound for the Castelnuovo-Mumford regularity of $\mathcal{S}$. This is why we assume $Y$ is a closed subscheme of a projective space bundle over $Z$.  This will allow use to eventually use the Gotzmann regularity theorem.

\begin{assumption}
In the statements below, we will assume that $B$ is a subscheme of a projective bundle  $\mathbf{P}(\mathcal{V}) = \operatorname{Proj}(\operatorname{Sym}^\bullet \mathcal{V})$ over $Z$ where $\mathcal{V}$ is flat and finite rank on $Z.$
\end{assumption}

To identify  $\mathcal{S}$ with an ideal sheaf, we first introduce an auxiliary projective space $K.$

\begin{defn}
Set
$$K = \operatorname{{Proj}} \ (\operatorname{Sym}^\bullet \mathcal{V})[E_1, \dotsc, E_s]$$ 
where $E$-variables are in degree $1$ 
(this is the projective closure of the product of $\mathbf{A}^{(s-1)}$ with the the affine cone over $\mathbf{P}(\mathcal{V})).$
\end{defn}

Now we define the space $N$ over which our ideal sheaf will live.

\begin{defn}
$B$ can be found as a subscheme of  the copy of $\mathbf{P}(\mathcal{V})$ in $K,$
cut out by 
$E_1 = \dotsm = E_s = 0.$
The first order infinitesimal neighborhood of $\mathbf{P}(\mathcal{V})$ is given by 
$$
N_{\mathbf{P}(\mathcal{V})} = \mathbf{V}(\langle E_1, \dotsc, E_s\rangle^2) \subseteq K.
$$
This is a scheme over $\mathbf{P}(\mathcal{V}),$ so we define
$$
N = B \times_{\mathbf{P}(\mathcal{V})} N_{\mathbf{P}(\mathcal{V})} \subseteq K.
$$  
\end{defn}

\begin{defn}
Consider the inclusion of $\Gamma_*(\mathcal{O}_K)$-modules
$$
(\Gamma_*(\mathcal{O}_K) / \Gamma_*(\mathcal{I}_B))^s(-1) \longrightarrow  \Gamma_*(\mathcal{O}_K) / \Gamma_*(\mathcal{I}_N)$$
which sends 
$$
(g_1, \dotsc, g_s)  \mapsto g_1E_1 + \dotsm +  g_1 E_s.
$$
The image of this map is $\Gamma_*(\mathcal{I}_B) / \Gamma_*(\mathcal{I}_N)$
The map induces an inclusion 
$$
\mathcal{O}_B(-1)^s \to \mathcal{O}_N.
$$
and an isomorphism $\mathcal{O}_B(-1)^s \cong \mathcal{I}_{B \subseteq N}.$
 $\mathcal{S}(-1)$ is carried to a
ideal of $\mathcal{O}_N$ we denote by $\mathcal{I}_{L \subseteq N},$  and we 
write  $L$ for the closed subscheme of $N$ defined by this  ideal.  
\end{defn}

This way we translate questions about $\mathcal{S}$ into questions about the 
ideal sheaf $\mathcal{I}_{L \subseteq N}.$

\begin{prop}
\label{prop:ideal-hilbert-polys}
$\pi_*\mathcal{S}(d) \cong \pi_*\mathcal{I}_{L \subseteq N}(d+1).$
\end{prop}
\begin{proof}
The $\mathcal{S}(-1) \to \mathcal{O}_N$ is an inclusion because $\tilde{(\cdot)}$ is exact.
\end{proof}

\begin{lem}
\label{lem:flat} 
If $B$ is $Z$-flat and $\mathcal{S}$ is a $Z$-family, then  $L$ and $N$ are $Z$-flat. 
\end{lem}
\begin{proof}
As $\mathcal{O}_B$-modules
$\mathcal{O}_N \cong \mathcal{O}_B \oplus \mathcal{O}_B(-1)^s,$
and
 $\mathcal{O}_L \cong \mathcal{O}_B \oplus \mathcal{G}(-1).$
 We know $\mathcal{G}$ is $Z$-flat by Lemma \ref{lem:flat-cok}.
\end{proof}

\begin{prop} 
If $B$ is  $Z$-flat and $\mathcal{S}$ is a $Z$-family, then
$\pi_* \mathcal{S}(d)$
is $Z$-flat of rank $p_\mathcal{S}(d)$ 
 provided 
that $d > $ the maximum of the  Castelnuovo-Mumford regularities of the ideal sheaves $\mathcal{I}_N$
and $\mathcal{I}_L.$ 
\end{prop}
\begin{proof}
Consider the direct image of the exact sequence
$$
0 \to \mathcal{I}_N(d+1) \to \mathcal{I}_L(d+1) \to \mathcal{I}_{L \subseteq N}(d+1) \to 0.
$$ 
Lemma \ref{lem:flat} and
the isomorphism $\mathcal{I}_{L \subseteq N}(d+1) \cong \mathcal{S}(d)$ imply 
these sheaves are $Z$-flat.
$d > $ the maximum of the  Castelnuovo-Mumford regularities of $\mathcal{I}_N$
and $\mathcal{I}_L,$ so we may apply the Cohomology and Base Change Theorem 
to conclude  $R^i \pi_* \mathcal{I}_{L \subseteq N}(d+1) =0$ for all $i >0,$
and $\pi_* \mathcal{I}_{L \subseteq N}(d+1)$ is $Z$-flat.
 
$\pi_*\mathcal{I}_N(d+1)$ and $\pi_*\mathcal{I}_L(d+1)$ are also $Z$-flat, so
$$
\operatorname{rank}(\pi_* \mathcal{S}(d)) = 
\operatorname{rank}(\pi_*\mathcal{I}_L(d+1))-
\operatorname{rank}(\pi_*\mathcal{I}_N(d+1)). 
$$
Since we are above the necessary Castelnuovo-Mumford regularities, the right-hand side equals  $p_{\mathcal{I}_L}(d+1) - p_{\mathcal{I}_N}(d+1).$
The relevant polynomials polynomials satisfy $p_\mathcal{S}(d) = p_{\mathcal{I}_{L \subseteq N}}(d+1) = p_{\mathcal{I}_L}(d+1) - p_{\mathcal{I}_N}(d+1),$ so $p_\mathcal{S}(d) = \operatorname{rank}(\pi_* \mathcal{S}(d)).$
\end{proof}

Now that we have formulated the regularity of $\mathcal{S}$ in terms of 
ideal sheaves on a projective space, we can use
Gotzmann regularity to give a bound for Castelnuovo-Mumford regularity of 
$\mathcal{S}.$  

\begin{lem}
\label{lem:hilbert-polynomials}
If $B$ is $Z$-flat, the Castelnuovo-Mumford regularities of 
$\mathcal{I}_N$
and $\mathcal{I}_L$ are bounded from above  by the maximum of the Gotzmann numbers of 
the 
Hilbert polynomials
of $\mathcal{O}_N$ and $\mathcal{O}_L$. 
Furthermore, the
  Hilbert polynomials of 
  $\mathcal{O}_N$ and $\mathcal{O}_L$
  can be expressed in terms of the Hilbert polynomials of $\mathcal{O}_B$ and $\mathcal{S}$ 
(or equivalently $\mathcal{O}_B$ and $\mathcal{G}$):
\begin{itemize}
\item $ \begin{array}{rcl} p_{\mathcal{O}_N}(t) & = & p_{\mathcal{O}_B}(t) + s \ p_{\mathcal{O}_B}(t-1), \end{array}$  and 
\item $\begin{array}{rcl}
p_{\mathcal{O}_L}(t) & = & p_{\mathcal{O}_N}(t) - p_\mathcal{S}(t-1)  \\
& = & p_{\mathcal{O}_B}(t) + s \, p_{\mathcal{O}_B}(t-1) - p_\mathcal{S}(t-1) \\
& = &  p_{\mathcal{O}_B}(t) + p_{\mathcal{G}}(t-1).
\end{array}$ 
\end{itemize}
\end{lem}
\begin{proof}
The first statement is part of the Gotzmann regularity theorem.
For the rest, we have the identifications
\begin{itemize}
\item  $\mathcal{O}_N \cong \mathcal{O}_B \oplus \mathcal{O}_B(-1)^s,$ and 
\item  $\mathcal{O}_L \cong \mathcal{O}_N / \mathcal{S}(-1)$
\end{itemize}
and the exact sequence
 $$
0 \to \mathcal{S} \to \mathcal{O}_B^s  \to \mathcal{G} \to 0.
 $$
\end{proof}

Finally, we can use this bound to guarantee the representability of the moduli of quasi-splines.

\begin{thm}
\label{thm:sections}
Assume $Y$ is $T$-flat and  is a closed subscheme of a projective space bundle $\mathbf{P}(\mathcal{V})$ over $T.$ Let $d$ be sufficiently large so that $\pi_*\mathcal{S}(d)$
is flat.  For instance, $d \geq$ the maximum of the Gotzmann numbers of the
polynomials 
$
p_{\mathcal{O}_Y}(t) + s \ p_{\mathcal{O}_Y}(t-1)
$ 
and
$
 p_{\mathcal{O}_Y}(t) + p_{\mathcal{G}}(t-1).
$
The functor on locally Noetherian $T$-schemes
$$
Z \mapsto \{ (\sigma, \mathcal{S}) \ | \ \sigma \in \Gamma(Z, \mathcal{S}(d)) \text{ and } \mathcal{S} \in {QS}_{p_{\mathcal{O}_Y},p_\mathcal{G}}(Z) \}
$$
is represented by  $\operatorname{{Spec}} (\operatorname{Sym}^\bullet \  (\pi_* \mathcal{S}(d))^\vee)$ over 
$
{QS}_{p_{\mathcal{O}_Y},p_\mathcal{G}}.
$
\end{thm}
\begin{proof}
$\pi_*\mathcal{S}(d)$ is locally free over 
${QS}_{p_{\mathcal{O}_Y},p_\mathcal{G}}.$  A morphism $Z \to 
\operatorname{{Spec}} (\operatorname{Sym}^\bullet \  (\pi_* \mathcal{S}(d))^\vee)$
produces a point $\mathcal{S} \in {QS}(Z)$ as well as a homomorphism over $Z$
$$
\mathcal{S}(d)^\vee \to  \mathcal{O}_Z. 
$$
Dually we have $\mathcal{O}_Z^\vee \to (\mathcal{S}(d)^\vee)^\vee.$ 
$\mathcal{O}_Z^\vee$ is canonically isomorphic to $\mathcal{O}_Z$
and $(\mathcal{S}(d)^\vee)^\vee$ is canonically isomorphic to 
$\mathcal{S}(d).$ So we obtain from the map 
$\mathcal{O}_Z \to \mathcal{S}(d)$ our global section $\sigma.$ 
This process is reversible,  so we are done.
\end{proof}

\appendix
\section{Billera-Rose Homogenization}

We recall the homogenization procedure of \cite{billera-rose:1991} in scheme theoretic language. 
Natural algebraic objects in homogenization and projective compactification are filtered algebras and modules. 
Given a quasi-spline sheaf $S$ over  $A = \operatorname{Spec} \mathcal{A},$ where $\mathcal{A}$
is a filtered algebra, this procedure produces
graded module  $\,^hS$ over the  homogenization $\widehat{\mathcal{A}}$ of $\mathcal{A},$
and a sheaf 
$\,^h\tilde{S}$ over the projective closure $\widehat{A} = \operatorname{Proj}(\widehat{\mathcal{A}}).$ 

The graded components of 
$\,^hS$ are isomorphic to the degree bounded pieces $S_{\leq d}$ of $S$.
Provided the homogenization $\widehat{\mathcal{A}}$ of $\mathcal{A}$ is isomorphic to
$\Gamma_*(\mathcal{O}_{\widehat{A}})$ these graded components  and degree bounded
pieces are isomorphic to the global sections $\Gamma(\widehat{A}, \,^h\tilde{S}(d)).$

These constructions are compatible with 
ideal difference-conditions in the sense that if $S$ is defined by $(I_{jk})_{jk},$ then 
\begin{itemize}
\item $\,^h{S}$ is defined by $(\,^h{I}_{jk})_{jk},$
\item $\,^h\tilde{S}$ is defined by $(\,^h\tilde{I}_{jk})_{jk},$ 
\item and $\Gamma_*(\,^h\tilde{S})$
is defined by $(\Gamma_*(\,^h\tilde{I}_{jk}))_{jk}.$
\end{itemize}

\begin{defn}
A filtered $\mathcal{O}_Z$-algebra  $\mathcal{A}$ is a
quasi-coherent sheaf of  $\mathcal{O}_Z$-algebras, equipped with 
a quasi-coherent $\mathcal{O}_Z$-submodule
$
\mathcal{A}_{\leq d}
$
 for each $d \in \mathbf{N}$
such that 
\begin{itemize}
\item $\mathcal{O}_Z \to \mathcal{A}_{\leq 0},$
\item $\mathcal{A}_{\leq d}
\subseteq \mathcal{A}_{\leq d+1},$  
\item $\bigcup_d \mathcal{A}_{\leq d} = \mathcal{A},$ and
\item $\mathcal{A}_{\leq d} \cdot \mathcal{A}_{\leq d'} \subseteq \mathcal{A}_{\leq d+d'}.$
\end{itemize}
\end{defn}

\begin{defn}
A filtered module over a filtered $\mathcal{O}_Z$-algebra $\mathcal{A}$ is a
quasi-coherent sheaf of  $\mathcal{A}$-modules $M$, 
equipped with a quasi-coherent $\mathcal{O}_Z$-submodule
$
{M}_{\leq d}
$
for each $d \in \mathbf{Z}$
such that 
\begin{itemize}
\item $M_{\leq d}
\subseteq M_{\leq d+1},$  
\item $\bigcup_d M_{\leq d} = M,$ and
\item $\mathcal{A}_{\leq d} \cdot M_{\leq d'} \subseteq M_{\leq d+d'}.$
\end{itemize}
\end{defn}

\begin{defn}
Given a filtered module over a filtered   $\mathcal{O}_Z$-algebra $\mathcal{A},$
we define the {\bf homogenization}
$$
\,^hM = \bigoplus_d M_{\leq d} \cdot z^d
$$
where $z$ is a ``dummy variable.''   Morally, we think of an element $m \in M$ as 
$m =  m(\frac{x_1}{z}, \dotsc, \frac{x_n}{z})$ where $\frac{x_1}{z}, \dotsc, \frac{x_n}{z}$
are (not-necessarily-algebraically-independent) ``coordinates'' in $\mathcal{A}.$
\end{defn}

\begin{defn}
We denote 
$\widehat{\mathcal{A}} = \,^h\mathcal{A}.$  This is a graded $\mathcal{O}_Z$-algebra, and within $\widehat{\mathcal{A}_1}$ there is an element $z = 1 \cdot z.$  If $N$ is a graded $\widehat{\mathcal{A}}$-module, we denote by $N|_{z=1}$
the module $N / \langle z -1 \rangle \cdot N.$  This module is filtered with
$$
(N|_{z=1})_{\leq d} = \text{ the image of $N_d$ under the quotient map}.
$$
\end{defn}

\begin{prop}
\label{prop:equiv}
Homogenization $M \mapsto \,^hM$ is an exact, fully faithful functor
from filtered  $\mathcal{A}$-modules to graded $\widehat{\mathcal{A}}$-modules. 
Furthermore the assignment $N \mapsto N|_{z=1}$ is a functor from 
graded $\widehat{\mathcal{A}}$-modules to filtered  $\mathcal{A}$-modules
which is left adjoint to homogenization.  The counit 
$$\epsilon \colon (\,^hM)|_{z=1} \to M$$
is a natural isomorphism, and the unit
$$
\eta \colon N \to \,^h(N|_{z=1})
$$
is surjective with kernel equal to the saturation $(0:z^\infty) \subseteq N.$

\end{prop}
\begin{proof}
These statements simply require checking definitions.  
\end{proof}

\begin{defn}
Let $\mathcal{A}$ be a quasi-coherent sheaf of filtered $\mathcal{O}_Z$-algebras.
Write $A = \operatorname{{Spec}} \mathcal{A},$  and
$\widehat{A} = \operatorname{{Proj}} \widehat{\mathcal{A}}$
for the {\bf projective closure} of $A$. 
\end{defn}

\begin{rem}
Our treatment differs only superficially from discussions, such as that in \cite{EGA_II}, on projective closures in which
$\mathcal{A}$ is presented as
$\mathcal{A} = \mathcal{T} / \mathcal{I}$ for a graded ring $\mathcal{T}$ and
a not-necessarily-graded ideal $\mathcal{I}.$  With such a presentation, one defines 
$\widehat{\mathcal{A}}$ as $\,^h\mathcal{T} / \,^h\mathcal{I},$  where $\mathcal{T}$
and $\mathcal{I}$ are given the filtration from the grading on $\mathcal{T}.$  This way, one 
only homogenizes submodules of graded modules.  
Even though 
there is no meaningful difference in these formulations, Proposition \ref{prop:equiv} becomes
awkward to state in terms of submodules of graded modules.
\end{rem}

\begin{defn} (\cite{billera-rose:1991})
 Let $S$ be a sheaf of quasi-splines on $A$. Equip $S$ with the filtration 
 $$
 S_{\leq d} = \{ (g_1, \dotsc, g_s) \in S \ | \ g_i \in \mathcal{A}_{\leq d} \text{ for all } i \}.
 $$
 We call $\,^hS$ the {\bf Billera-Rose homogenization} of $S$.
 \end{defn}

\begin{lem}
$\,^hS$ is a quasi-spline sheaf on $\operatorname{Spec}(\widehat{\mathcal{A}}).$
\end{lem}
\begin{proof}
Omitted.
\end{proof}

One source of the usefulness of homogenization is that it identifies 
degree $d$ bounded elements with degree $d$ homogeneous elements.

\begin{lem}
\label{lem:bijection-lemma}
As $\mathcal{O}_Z$-modules,
$M_{\leq d}$ is isomorphic to  $\,^hM_d.$
\end{lem}
\begin{proof}
Omitted.
\end{proof}

Although, $\,^hS$ defines a quasi-spline sheaf $\,^h\tilde{S}$ on the projective closure $\widehat{A}$, it is not always the case that $\,^hS$  can be recovered from $\,^h\tilde{S}.$  However, under mild conditions on $\mathcal{A}$  it can, and  in this situation questions about $S$ to be completely translated into questions in projective geometry.

\begin{prop}
\label{prop:saturation}
If $\widehat{\mathcal{A}} \to \Gamma_*(\mathcal{O}_{\widehat{A}})$
is an isomorphism, then
\begin{itemize}
\item $\,^hS \to \Gamma_*(\,^h\tilde{S})$ for any quasi-spline sheaf $S$ over $A$, and
\item $\,^hI \to \Gamma_*(\,^h\tilde{I})$ for any quasi-coherent $\mathcal{A}$-ideal $I$
\end{itemize}
are too.
\end{prop}
\begin{proof}
$\Gamma_*$ is left exact, so we have inclusions 
$\Gamma_*(\,^h\tilde{S}) \to \widehat{\mathcal{A}}^s$ and 
$\Gamma_*(\,^h\tilde{I})  \to \widehat{\mathcal{A}}.$ 
The inclusions $\,^hS \to \widehat{\mathcal{A}}^s$ and 
$\,^hI \to \widehat{\mathcal{A}}^s$ factor though these maps.
$z \in  \widehat{\mathcal{A}}_1,$ and  
since $z$ is in degree $1,$
restricting to $D_+(z) \subseteq \widehat{A}$ has the same effect as setting $z=1.$
Setting $z=1$ carries $\,^h\tilde{S}$ to $S$ and 
$\,^h\tilde{I}$ to $I,$  so 
$\Gamma_*(\,^h\tilde{S})$ is carried onto $S$ and 
$\Gamma_*(\,^h\tilde{I})$  is carried onto $I.$ 

Now observe that $z$ is a non-zero divisor on  $\widehat{\mathcal{A}}$ and 
$\widehat{\mathcal{A}}^s,$ and thus a non-zero divisor on any submodule of these.
So Proposition \ref{prop:equiv} implies both $\Gamma_*(\,^h\tilde{S})$ and $\,^hS$ equal $\,^hS,$
and both  $\Gamma_*(\,^h\tilde{I})$ and $\,^hI$ equal  $\,^hI$. 
\end{proof}

\begin{rem} It is not always the case that
$\widehat{\mathcal{A}} \to \Gamma_*(\mathcal{O}_{\widehat{A}})$
is an isomorphism.  For instance, consider 
$$
\mathcal{A} = \mathbf{C}[x,y] /  \langle xy, y^2 \rangle
$$
with the degree filtration from  $\mathbf{C}[x,y]$.  Then $\widehat{\mathcal{A}} = \mathbf{C}[x,y, z] /  \langle xy, y^2 \rangle$
and $\Gamma_*(\mathcal{O}_{\widehat{A}})  = \mathbf{C}[\sigma][x,y, \sigma, z] /  \langle \sigma x, \sigma^2, \sigma z - y. \rangle.$  A representative of $\sigma$ in the \v{C}ech cohomology with respect to the cover $\{U_x, U_z\}$
is 
$$
\sigma = (0,  \frac{y}{z}) \in 
\mathbf{C}[\frac{z}{x}, \frac{y}{x}] / \langle \frac{y}{x} \rangle \times 
\mathbf{C}[\frac{x}{z}, \frac{y}{z}]/\langle  \frac{x}{z}\frac{y}{z}, \frac{y^2}{z^2}\rangle. 
$$
Counterexamples can also be found by considering the scheme $A$ to be the complement of a hypersurface on a non-projectively-normal variety.
\end{rem}

If $S$ is computed from ideal difference-conditions  both the homogenization $\,^hS$ of
$S$ and and the sheaf $\,^h\tilde{S}$ on the projective closure $\widehat{A}$ can be 
computed from the associated ideal difference-conditions.

\begin{prop}
If $S$ is defined by the ideal difference-conditions $(I_{jk})_{jk},$
then 
\begin{itemize}
\item $\,^hS$ is defined by the ideal difference-conditions 
 $(\,^hI_{jk})_{jk},$  and
 \item $\,^h\tilde{S}$ is defined by 
 the ideal difference-conditions $(\,^h\tilde{I}_{jk})_{jk}.$
 \end{itemize}
 \end{prop}
 \begin{proof}
The first statement is an immediate consequence of Proposition \ref{prop:equiv}. 
 The second  follows from the first and the exactness of $\tilde{(\cdot)}:$ localization is exact, and popping out the $0^\text{th}$ graded piece is exact. 
 \end{proof}

The preceding results establish what is needed from the Billera-Rose homogenization 
to use it as a tool for studying quasi-splines over affine schemes
using projective geometric techniques.  However, in what 
is in some sense the opposite direction, we include the
 following observation relating quasi-splines on projective schemes
defined by ideal difference-conditions and those on their affine cones.

\begin{prop}
Let $\mathcal{S}$ be a quasi-spline sheaf over a projective $Z$-scheme $B$
defined by the ideal difference-conditions $(\mathfrak{I}_{jk})_{jk}.$
Then 
$
\Gamma_*(\mathcal{S})
$
is a module of quasi-splines over the $Z$-affine cone
$
\operatorname{{Spec}} \Gamma_*(\mathcal{O}_B)
$
defined by the ideal difference conditions 
$
( \Gamma_*(\mathfrak{I}_{jk}))_{jk}.
$
\end{prop}
\begin{proof}
$\Gamma_*$ is left exact.
\end{proof}

\subsection*{Acknowledgements}   
The author would like to thank F. Sottile and D. Cox for their encouragement and  comments on an early version of this manuscript.   
The main ideas of this paper were sorted out while the author was at the IH\'ES under the support of National Science Foundation Grant No. 1002477.  We appreciate the the wonderful working environment and hospitality l'Institut provides.
Finally, thanks to R. Perline for explaining B\'ezier-Bernstein methods to us, and L. Lapointe for helping translate the Abstract into French.

\newpage
\bibliography{spline_moduli}

\begin{thebibliography}{BDCP90}

\bibitem[AS87]{alfeld-schumaker-1987}
Peter Alfeld and L.~L. Schumaker.
\newblock The dimension of bivariate spline spaces of smoothness {$r$} for
  degree {$d\geq 4r+1$}.
\newblock {\em Constr. Approx.}, 3(2):189--197, 1987.

\bibitem[AS90]{alfeld-schumaker-1990}
Peter Alfeld and Larry~L. Schumaker.
\newblock On the dimension of bivariate spline spaces of smoothness {$r$} and
  degree {$d=3r+1$}.
\newblock {\em Numer. Math.}, 57(6-7):651--661, 1990.

\bibitem[ASW93]{alfeld-schumaker-whiteley}
Peter Alfeld, Larry~L. Schumaker, and Walter Whiteley.
\newblock The generic dimension of the space of {$C^1$} splines of degree
  {$d\geq 8$} on tetrahedral decompositions.
\newblock {\em SIAM J. Numer. Anal.}, 30(3):889--920, 1993.

\bibitem[BDCP90]{bifet-deconcini-procesi}
Emili Bifet, Corrado De~Concini, and Claudio Procesi.
\newblock Cohomology of regular embeddings.
\newblock {\em Adv. Math.}, 82(1):1--34, 1990.

\bibitem[Ber12]{bernstein-1912}
Serge\"i~Natanovitch Bernstein.
\newblock D\'emonstration du th\'eor\`eme de {W}eierstrauss, fond\'ee sur le
  calcul des probabilit\'es.
\newblock {\em Comm. Soc. Math. Kharkov}, 13, 1912.

\bibitem[BH93]{bruns-herzog}
Winfried Bruns and J{\"u}rgen Herzog.
\newblock {\em Cohen-{M}acaulay rings}, volume~39 of {\em Cambridge Studies in
  Advanced Mathematics}.
\newblock Cambridge University Press, Cambridge, 1993.

\bibitem[Bil88]{billera-1988}
Louis~J. Billera.
\newblock Homology of smooth splines: generic triangulations and a conjecture
  of {S}trang.
\newblock {\em Trans. Amer. Math. Soc.}, 310(1):325--340, 1988.

\bibitem[BR89]{billera-rose:1989}
L.~J. Billera and L.~L. Rose.
\newblock Gr\"obner basis methods for multivariate splines.
\newblock In {\em Mathematical methods in computer aided geometric design
  ({O}slo, 1988)}, pages 93--104. Academic Press, Boston, MA, 1989.

\bibitem[BR91]{billera-rose:1991}
Louis~J. Billera and Lauren~L. Rose.
\newblock A dimension series for multivariate splines.
\newblock {\em Discrete Comput. Geom.}, 6(2):107--128, 1991.

\bibitem[BR92]{billera-rose-modules}
Louis~J. Billera and Lauren~L. Rose.
\newblock Modules of piecewise polynomials and their freeness.
\newblock {\em Math. Z.}, 209(4):485--497, 1992.

\bibitem[Bri97]{brion-1997}
M.~Brion.
\newblock Equivariant {C}how groups for torus actions.
\newblock {\em Transform. Groups}, 2(3):225--267, 1997.

\bibitem[Cou43]{courant-1943}
R.~Courant.
\newblock Variational methods for the solution of problems of equilibrium and
  vibrations.
\newblock {\em Bull. Amer. Math. Soc.}, 49:1--23, 1943.

\bibitem[dC59]{deCasteljau-1959}
P.~de~Casteljau.
\newblock {Outillage m\'ethods calcul}.
\newblock {\em {Paris, Andr\'e Citro\"en Automobiles SA}}, 1959.

\bibitem[Far77]{farin-1977}
Gerald~E. Farin.
\newblock {\em Konstruktion und Eigenschaften von B\'ezier-Kurven und B\'ezier
  Fl\"achen}.
\newblock PhD thesis, University of Braunschweig, 1977.

\bibitem[FS13]{foucart-sorokina-2013}
Simon Foucart and Tatyana Sorokina.
\newblock Generating dimension formulas for multivariate splines.
\newblock {\em Albanian J. Math.}, 7(1):25--35, 2013.

\bibitem[GKM98]{goresky-kottwitz-macpherson-1998}
Mark Goresky, Robert Kottwitz, and Robert MacPherson.
\newblock Equivariant cohomology, {K}oszul duality, and the localization
  theorem.
\newblock {\em Invent. Math.}, 131(1):25--83, 1998.

\bibitem[Got78]{gotzmann-1978}
Gerd Gotzmann.
\newblock Eine {B}edingung f\"ur die {F}lachheit und das {H}ilbertpolynom eines
  graduierten {R}inges.
\newblock {\em Math. Z.}, 158(1):61--70, 1978.

\bibitem[GPT13]{gilbert-polster-tymoczko}
Simcha Gilbert, Shira Polster, and Julianna Tymoczko.
\newblock Generalized splines on arbitrary graphs.
\newblock 2013, 0902.0885.
\newblock arXiv:1306.0801.

\bibitem[Gro61]{grothendieck-hilbert}
Alexander Grothendieck.
\newblock Techniques de construction et th\'eor\`emes d'existence en
  g\'eom\'etrie alg\'ebrique. {IV}. {L}es sch\'emas de {H}ilbert.
\newblock In {\em S\'eminaire {B}ourbaki, {V}ol.\ 6}, pages Exp.\ No.\ 221,
  249--276. Paris, 1960-61.

\bibitem[Gro61]{EGA_II}
A.~Grothendieck.
\newblock \'{E}l\'ements de g\'eom\'etrie alg\'ebrique. {II}. \'{E}tude globale
  \'el\'ementaire de quelques classes de morphismes.
\newblock {\em Inst. Hautes \'Etudes Sci. Publ. Math.}, (8):222, 1961.

\bibitem[Gro63]{EGA_III_2}
A.~Grothendieck.
\newblock \'{E}l\'ements de g\'eom\'etrie alg\'ebrique. {III}. \'{E}tude
  cohomologique des faisceaux coh\'erents. {II}.
\newblock {\em Inst. Hautes \'Etudes Sci. Publ. Math.}, (17):91, 1963.

\bibitem[GS97]{geramita-schenk-1997}
A.~V. Geramita and H.~Schenck.
\newblock Fat points, inverse systems and piecewise polynomial functions.
\newblock In {\em The {C}urves {S}eminar at {Q}ueen's, {V}ol.\ {XI}
  ({K}ingston, {ON}, 1997)}, volume 105 of {\em Queen's Papers in Pure and
  Appl. Math.}, pages 98--116. Queen's Univ., Kingston, ON, 1997.

\bibitem[Har77]{hartshorne-1977}
Robin Hartshorne.
\newblock {\em Algebraic geometry}.
\newblock Springer-Verlag, New York, 1977.
\newblock Graduate Texts in Mathematics, No. 52.

\bibitem[Hre41]{hrennikoff-1941}
A.~Hrennikoff.
\newblock Solution of problems of elasticity by the framework method.
\newblock {\em J. Appl. Mech.}, 8:A--169--A--175, 1941.

\bibitem[Iar97]{iarrobino-1997}
A.~Iarrobino.
\newblock Inverse system of a symbolic power. {III}. {T}hin algebras and fat
  points.
\newblock {\em Compositio Math.}, 108(3):319--356, 1997.

\bibitem[LS07]{lai-schumaker}
Ming-Jun Lai and Larry~L. Schumaker.
\newblock {\em Spline functions on triangulations}, volume 110 of {\em
  Encyclopedia of Mathematics and its Applications}.
\newblock Cambridge University Press, Cambridge, 2007.

\bibitem[M{\"o}b27]{mobius:1827}
August~Ferdinand M{\"o}bius.
\newblock {\em Der barycentrische {C}alcul}.
\newblock J. A. Barth Verlag, Leipzig, 1827.

\bibitem[Mum66]{mumford-curves}
David Mumford.
\newblock {\em Lectures on curves on an algebraic surface}.
\newblock With a section by G. M. Bergman. Annals of Mathematics Studies, No.
  59. Princeton University Press, Princeton, N.J., 1966.

\bibitem[Pay06]{payne-2006}
Sam Payne.
\newblock Equivariant {C}how cohomology of toric varieties.
\newblock {\em Math. Res. Lett.}, 13(1):29--41, 2006.

\bibitem[Sch46]{schoenberg-1946}
I.~J. Schoenberg.
\newblock Contributions to the problem of approximation of equidistant data by
  analytic functions. {P}art {A}. {O}n the problem of smoothing or graduation.
  {A} first class of analytic approximation formulae.
\newblock {\em Quart. Appl. Math.}, 4:45--99, 1946.

\bibitem[Sch79]{schumaker:1979}
Larry~L. Schumaker.
\newblock On the dimension of spaces of piecewise polynomials in two variables.
\newblock In {\em Multivariate approximation theory ({P}roc. {C}onf., {M}ath.
  {R}es. {I}nst., {O}berwolfach, 1979)}, volume~51 of {\em Internat. Ser.
  Numer. Math.}, pages 396--412. Birkh\"auser, Basel, 1979.

\bibitem[Sch97]{schenck:1997}
Hal Schenck.
\newblock A spectral sequence for splines.
\newblock {\em Adv. in Appl. Math.}, 19(2):183--199, 1997.

\bibitem[Sch11]{schenck-2011}
Hal Schenck.
\newblock Equivariant chow cohomology of nonsimplicial toric varieties.
\newblock Transactions of the A.M.S., 364 (2012) 4041-4051, 2011,
  arXiv:1101.0352.

\bibitem[Sch14]{schenck-2014}
Hal Schenck.
\newblock Splines on the {A}lfeld split of a simplex and type {A} root systems.
\newblock {\em J. Approx. Theory}, 182:1--6, 2014.

\bibitem[Ser55]{serre:1955}
Jean-Pierre Serre.
\newblock Faisceaux alg\'ebriques coh\'erents.
\newblock {\em Ann. of Math. (2)}, 61:197--278, 1955.

\bibitem[SS97]{schenck-stillman:1997}
Hal Schenck and Mike Stillman.
\newblock Local cohomology of bivariate splines.
\newblock {\em J. Pure Appl. Algebra}, 117/118:535--548, 1997.
\newblock Algorithms for algebra (Eindhoven, 1996).

\bibitem[SS02]{Schenck-Stiller-2002}
Henry~K. Schenck and Peter~F. Stiller.
\newblock Cohomology vanishing and a problem in approximation theory.
\newblock {\em Manuscripta Math.}, 107(1):43--58, 2002.

\bibitem[{Sta}13]{stacks-project}
The {Stacks Project Authors}.
\newblock {\itshape Stacks Project}.
\newblock \url{http://stacks.math.columbia.edu}, 2013.

\bibitem[Sti83]{stiller:1983}
Peter~F. Stiller.
\newblock Certain reflexive sheaves on {${\bf P}^{n}_{{\bf C}}$} and a problem
  in approximation theory.
\newblock {\em Trans. Amer. Math. Soc.}, 279(1):125--142, 1983.

\bibitem[Str74]{strang:1974}
Gilbert Strang.
\newblock The dimension of piecewise polynomial spaces, and one-sided
  approximation.
\newblock In {\em Conference on the {N}umerical {S}olution of {D}ifferential
  {E}quations ({U}niv. {D}undee, {D}undee, 1973)}, pages 144--152. Lecture
  Notes in Math., Vol. 363. Springer, Berlin, 1974.

\bibitem[Wei85]{weierstrass-1885}
K.~Weierstrass.
\newblock {\"Uber die analytische Darstellbarkeit sogenannter willk\"urlicher
  Functionen einer reellen Ver\"anderlichen}.
\newblock {\em {Sitzungsberichte der K\"oniglich Preu\ss ischen Akademie der
  Wissenschaften zu Berlin}}, (II), 1885.

\bibitem[Yuz92]{yuzvinsky:1992}
Sergey Yuzvinsky.
\newblock Modules of splines on polyhedral complexes.
\newblock {\em Math. Z.}, 210(2):245--254, 1992.

\end{thebibliography}
\bibliographystyle{halpha}  

\noindent
{Department of Mathematics, Drexel University, Philadelphia, PA 19104\\
\texttt{pclarke@math.drexel.edu}}

\end{document}